\documentclass[11pt,oneside,english,reqno]{amsart}
\usepackage[T1]{fontenc}
\usepackage[latin9]{inputenc}
\usepackage{geometry}
\usepackage{babel}
\usepackage{float}
\usepackage{mathrsfs}
\usepackage{amsthm}
\usepackage{amstext}
\usepackage{amssymb}
\usepackage{graphicx}
\usepackage{setspace}
\usepackage{esint}
\PassOptionsToPackage{normalem}{ulem}
\usepackage{ulem}
\usepackage[unicode=true,pdfusetitle,
 bookmarks=true,bookmarksnumbered=false,bookmarksopen=false,
 breaklinks=false,pdfborder={0 0 1},backref=false,colorlinks=false]
 {hyperref}
\hypersetup{
 colorlinks=true,citecolor=blue,linkcolor=blue,linktocpage=true}

\makeatletter

\providecommand{\tabularnewline}{\\}

\numberwithin{equation}{section}
\numberwithin{figure}{section}
\numberwithin{table}{section}
\usepackage{enumitem}		
\theoremstyle{plain}
\newtheorem{thm}{\protect\theoremname}[section]
  \theoremstyle{definition}
  \newtheorem{defn}[thm]{\protect\definitionname}
  \theoremstyle{plain}
  \newtheorem{lem}[thm]{\protect\lemmaname}
  \theoremstyle{remark}
  \newtheorem{rem}[thm]{\protect\remarkname}
  \theoremstyle{plain}
  \newtheorem{cor}[thm]{\protect\corollaryname}
  \theoremstyle{definition}
  \newtheorem{example}[thm]{\protect\examplename}
  \theoremstyle{plain}
  \newtheorem{question}[thm]{\protect\questionname}
  \theoremstyle{plain}
  \newtheorem{prop}[thm]{\protect\propositionname}
  \theoremstyle{remark}
  \newtheorem*{acknowledgement*}{\protect\acknowledgementname}

\providecommand{\MR}[1]{}

\renewcommand{\section}{%
\@startsection{section}{1}%
  \z@{.7\linespacing\@plus\linespacing}{.5\linespacing}%
  {\normalfont\scshape\centering\bfseries}}

\renewcommand{\subsection}{%
\@startsection{subsection}{2}%
  \z@{.5\linespacing\@plus.7\linespacing}{.5\linespacing}%
  {\normalfont\bfseries}}

\renewcommand{\subsubsection}{%
\@startsection{subsubsection}{2}%
  \z@{.5\linespacing\@plus.7\linespacing}{.5\linespacing}%
  {\normalfont\bfseries}}

\usepackage{bbm}
\usepackage{circuitikz}

\AtBeginDocument{
  
}

\makeatother

  \providecommand{\acknowledgementname}{Acknowledgement}
  \providecommand{\corollaryname}{Corollary}
  \providecommand{\definitionname}{Definition}
  \providecommand{\examplename}{Example}
  \providecommand{\lemmaname}{Lemma}
  \providecommand{\propositionname}{Proposition}
  \providecommand{\questionname}{Question}
  \providecommand{\remarkname}{Remark}
\providecommand{\theoremname}{Theorem}

\begin{document}

\title{Infinite weighted graphs with bounded resistance metric}

\author{Palle Jorgensen and Feng Tian}

\address{(Palle E.T. Jorgensen) Department of Mathematics, The University
of Iowa, Iowa City, IA 52242-1419, U.S.A. }

\email{palle-jorgensen@uiowa.edu}

\urladdr{http://www.math.uiowa.edu/\textasciitilde{}jorgen/}

\address{(Feng Tian) Department of Mathematics, Trine University, IN 46703,
U.S.A.}

\email{tianf@trine.edu}



\dedicatory{To the memory of Ola Bratteli}
\begin{abstract}
We consider infinite weighted graphs $G$, i.e., sets of vertices
$V$, and edges $E$ assumed countable infinite. An assignment of
weights is a positive symmetric function $c$ on $E$ (the edge-set),
conductance. From this, one naturally defines a reversible Markov
process, and a corresponding Laplace operator acting on functions
on $V$, voltage distributions. The harmonic functions are of special
importance. We establish explicit boundary representations for the
harmonic functions on $G$ of finite energy.

We compute a resistance metric $d$ from a given conductance function.
(The resistance distance $d\left(x,y\right)$ between two vertices
$x$ and $y$ is the voltage drop from $x$ to $y$, which is induced
by the given assignment of resistors when 1 amp is inserted at the
vertex $x$, and then extracted again at $y$.)

We study the class of models where this resistance metric is bounded.
We show that then the finite-energy functions form an algebra of $\frac{1}{2}$-Lipschitz-continuous
and bounded functions on $V$, relative to the metric $d$. We further
show that , in this case, the metric completion $M$ of $\left(V,d\right)$
is automatically compact, and that the vertex-set $V$ is open in
$M$. We obtain a Poisson boundary-representation for the harmonic
functions of finite energy, and an interpolation formula for every
function on $V$ of finite energy.

We further compare $M$ to other compactifications; e.g., to certain
path-space models.
\end{abstract}

\maketitle
\tableofcontents{}

\section{Introduction}

We consider a certain class of infinite weighted graphs $G$. They
are specified by prescribed sets of vertices $V$, and edges $E$;
countable infinite. An assignment of weights, is a positive symmetric
function $c$ of $E$ (the edge-set). In electrical network models,
the function $c$ represents conductance, and its reciprocal resistance.
So fixing a conductance function is then equivalent to an assignment
of resistors on the edges of $G$. From this, one naturally defines
a reversible Markov process, and a corresponding Laplace operator
(called graph Laplacian) acting on functions on $V$, the vertex-set.
Functions on $V$ typically represent voltage distributions, and the
harmonic functions are of special importance. For list of explicit
details required on $\left(V,E,c\right)$, we refer to the details
in Section \ref{sub:setting}.

We will be especially interested in boundary representations for harmonic
functions of finite energy.

From a given conductance function, we compute a resistance metric
$d$. Intuitively, the resistance distance $d\left(x,y\right)$ between
two vertices $x$ and $y$ is the voltage drop from $x$ to $y$,
which is induced by the given assignment of resistors when 1 amp is
inserted at the vertex $x$, and then extracted again at $y$. We
study the realistic class of models when this resistance metric is
assumed bounded. In this case the finite-energy functions form an
algebra of continuous and bounded functions on $V$, relative to the
metric $d$. We further show that, in this case, the metric completion
$M$ of $\left(V,d\right)$ is automatically compact. The vertex-set
$V$ is open in $M$, and we obtain a Poisson boundary-representation
for the harmonic functions of finite energy.

A number of additional properties are established for $M$. In particular,
we compare $M$ to other compactifications in the literature; e.g.,
to path-space models.

There is a recent increased interest in analysis on large (infinite)
networks, motivated by a host of applications; see e.g., \cite{JP10,JP11,AJV14,MR2905788,MR3025713}.
We shall be citing standard facts from the general theory. In addition,
we use facts from analysis, Hilbert space geometry, potential theory,
boundaries, and Markov measures; see e.g., \cite{MR3039515,MR3290453,MR3046303,MR2982692,MR2811284,MR2815030,MR3150704}.

\section{\label{sub:setting}Basic settings}

Let $G=\left(V,E,c\right)$ be a weighted graph, where $c=$ conductance
function (see Definition \ref{def:cond}), $V=$ vertex-set (countable
infinite), and the edges $E\subset V\times V\backslash\left\{ \text{diagonal}\right\} $
such that:
\begin{enumerate}[label=(G\arabic{enumi}),ref=G\arabic{enumi}]
\item $\left(x,y\right)\in E\Longleftrightarrow\left(y,x\right)\in E$;
$x,y\in V$;
\item $0<\#\left\{ y\in V\:|\:\left(x,y\right)\in E\right\} <\infty$, for
all $x\in V$; 
\item \label{enu:net3}Connectedness: $\exists\, o\in V$ s.t. for all $y\in V$
$\exists\, x_{0},x_{1},\ldots,x_{n}\in V$ with $x_{0}=o$, $x_{n}=y$,
$\left(x_{i-1},x_{i}\right)\in E$, $\forall i=1,\ldots,n$. 
\item \label{enu:net4}If a conductance function $c$ is given, we require
$c_{x_{i-1}x_{i}}>0$. \end{enumerate}
\begin{defn}
\label{def:cond}A function $c:E\rightarrow\mathbb{R}_{+}\cup\left\{ 0\right\} $
is called \emph{conductance function} if 
\begin{enumerate}
\item $c\left(e\right)\geq0$, $\forall e\in E$;
\item Given $x\in V$, $c_{xy}>0$, $c_{xy}=c_{yx}$, for all $\left(xy\right)\in E$;
\item If $\left(x,y\right)\in E$, we write $x\sim y$; and it is assumed
that $\#\left\{ y\in V\:|\: y\sim x\right\} $ is finite for all $x\in V$.
\end{enumerate}

If $x\in V$, we set 
\begin{align}
c\left(x\right) & :=\sum_{x\sim y}c_{xy},\quad\mbox{where \ensuremath{x\sim y\underset{\text{Def}}{\Longleftrightarrow}\left(xy\right)\in E}}.\label{eq:cond}
\end{align}

\end{defn}

Let $G=\left(V,E,c\right)$ be as above. Assume $G$ is connected,
i.e., there is a base point $o$ in $V$ such that every $x\in V$
is connected to $o$ via a finite path of edges, see (\ref{enu:net3}). 

Set $\mathscr{H}_{E}:=$ completion of functions $u:V\rightarrow\mathbb{C}$
with respect to 
\begin{align}
\left\langle u,v\right\rangle _{\mathscr{H}_{E}} & :=\frac{1}{2}\underset{\left(x,y\right)\in E}{\sum\sum}c_{xy}(\overline{u\left(x\right)}-\overline{u\left(y\right)})\left(v\left(x\right)-v\left(y\right)\right)\label{eq:Einn}\\
\left\Vert u\right\Vert _{\mathscr{H}_{E}}^{2} & :=\frac{1}{2}\underset{\left(x,y\right)\in E}{\sum\sum}c_{xy}\left|u\left(x\right)-u\left(y\right)\right|^{2}\label{eq:Enorm}
\end{align}
(or simply all functions $u$ s.t. the sum in (\ref{eq:Enorm}) is
finite.) Then $\mathscr{H}_{E}$ is a Hilbert space \cite{JP10}.
($\mathscr{H}_{E}$ is known to be bigger than the $\mathscr{H}_{E}$-norm
completion of the finitely supported functions on $V$. We know that
the non-constant harmonic functions on $V$ are \emph{not} in the
$\mathscr{H}_{E}$-completion of the finitely supported functions,
see Remark \ref{rem:sp}.)
\begin{lem}[\cite{JP10}]
(i) For every pair of vertices $x,y\in V$, there is a unique $v_{x,y}\in\mathscr{H}_{E}$
(unique up to an additive constant) such that 
\begin{equation}
f\left(x\right)-f\left(y\right)=\left\langle v_{xy},f\right\rangle _{\mathscr{H}_{E}},\quad\forall f\in\mathscr{H}_{E}.\label{eq:dd1}
\end{equation}

(ii) The vector $v_{xy}$ in (\ref{eq:dd1}) satisfies 
\begin{equation}
\Delta v_{xy}=\delta_{x}-\delta_{y},\label{eq:dd2}
\end{equation}
where $\left(\Delta f\right)\left(u\right):=\sum_{y\sim u}c_{uy}\left(f\left(u\right)-f\left(y\right)\right).$\end{lem}
\begin{rem}
The solution to (\ref{eq:dd2}) is not unique: If $v_{xy}$ satisfies
(\ref{eq:dd2}), and if $h\in\mathscr{H}_{E}$ satisfies $\Delta h=0$
(harmonic), then $v_{xy}+h$ also satisfies (\ref{eq:dd2}); but generally
\emph{not} (\ref{eq:dd1}).
\end{rem}
Let $V':=V\backslash\left\{ o\right\} $, and set
\begin{align}
v_{x} & :=v_{x,o},\quad\forall x\in V'.\label{eq:vx}
\end{align}

\begin{cor}
\label{lem:dipole}For all $x,y\in V$, $\exists!$ real-valued \uline{dipole}
vector $v_{xy}\in\mathscr{H}_{E}$ s.t.
\begin{align}
\left\langle v_{xy},u\right\rangle _{\mathscr{H}_{E}} & =u\left(x\right)-u\left(y\right),\quad\forall u\in\mathscr{H}_{E}.\label{eq:dipole}
\end{align}
Moreover, 
\begin{equation}
v_{xy}-v_{zy}=v_{xz},\quad\forall x,y,z\in V.
\end{equation}
\end{cor}
\begin{defn}
\label{def:lap}Fix a weighted graph (connected), set the \emph{graph
Laplacian} $\Delta=\Delta_{c}$, where
\begin{align}
\left(\Delta u\right)\left(x\right) & =\sum_{y\sim x}c_{xy}\left(u\left(x\right)-u\left(y\right)\right)=c\left(x\right)u\left(x\right)-\sum_{y\sim x}c_{xy}u\left(y\right)\label{eq:lap}
\end{align}
for all functions $u$ on $V$.
\end{defn}
Lemma \ref{lem:Delta} below summarizes the key properties of $\Delta$
as an operator, both in $l^{2}(V)$ and in $\mathscr{H}_{E}$. 
\begin{defn}
Let $\left(V,E,c\right)$ and $\Delta$ be as outlined, and let $\mathscr{H}_{E}$
be the corresponding energy-Hilbert space; see (\ref{eq:Enorm}).
Finally let $l^{2}=l^{2}\left(V\right)$ denote the usual $l^{2}$-space,
i.e., all $w:V\rightarrow\mathbb{C}$ such that 
\[
\left\Vert w\right\Vert _{l^{2}}^{2}=\sum_{x\in V}\left|w\left(x\right)\right|^{2}<\infty.
\]
We shall need the subspace $\mathscr{D}_{2}\subset l^{2}$ (dense
in the $l^{2}$-norm):
\begin{equation}
\mathscr{D}_{2}:=span\left\{ \delta_{x}\:\big|\: x\in V\right\} .
\end{equation}
If $\left\{ v_{x}\:|\: x\in V'\right\} $ denotes a system of dipoles
(see (\ref{eq:vx})), we set $\mathscr{D}_{E}\subset\mathscr{H}_{E}$
(dense in $\mathscr{H}_{E}$-norm):
\begin{equation}
\mathscr{D}_{E}:=span\left\{ v_{x}\:\big|\: x\in V'\right\} ;\label{eq:domE}
\end{equation}
in both cases ``span'' means all finite linear combinations. 

We show in Section \ref{sec:path} that $l^{2}\left(V\right)$ contains
\emph{no} non-constant harmonic functions; but $\mathscr{H}_{E}$
generally does.\end{defn}
\begin{lem}
\label{lem:Delta}The following hold:
\begin{enumerate}
\item \label{enu:D1}$\left\langle \Delta u,v\right\rangle _{l^{2}}=\left\langle u,\Delta v\right\rangle _{l^{2}}$,
$\forall u,v\in\mathscr{D}_{2}$;
\item \label{enu:D2}$\left\langle \Delta u,v\right\rangle _{\mathscr{H}_{E}}=\left\langle u,\Delta v\right\rangle _{\mathscr{H}_{E}},$
$\forall u,v\in\mathscr{D}_{E}$;
\item \label{enu:D3}$\left\langle u,\Delta u\right\rangle _{l^{2}}\geq0$,
$\forall u\in\mathscr{D}_{2}$, and
\item $\left\langle u,\Delta u\right\rangle _{\mathscr{H}_{E}}\geq0$, $\forall u\in\mathscr{D}_{E}$.
\end{enumerate}

As a densely defined operator in $l^{2}\left(V\right)$, $\Delta$
is essentially selfadjoint; but, as an operator with dense domain
in $\mathscr{H}_{E}$, $\Delta$ is generally \uline{not} essentially
selfadjoint.

Moreover, we have
\begin{enumerate}[resume]
\item \label{enu:D5}$\left\langle \delta_{x},u\right\rangle _{\mathscr{H}_{E}}=\left(\Delta u\right)\left(x\right)$,
$\forall x\in V$, $\forall u\in\mathscr{H}_{E}$.
\item $\Delta v_{xy}=\delta_{x}-\delta_{y}$, $\forall x,y\in V$, where
$v_{xy}\in\mathscr{H}_{E}$. In particular, $\Delta v_{x}=\delta_{x}-\delta_{o}$,
$x\in V'=V\backslash\left\{ o\right\} $. 
\item \label{enu:D7}
\begin{align*}
\delta_{x}\left(\cdot\right) & =c\left(x\right)v_{x}\left(\cdot\right)-\sum_{y\sim x}c_{xy}v_{y}\left(\cdot\right),\;\forall x\in V'.
\end{align*}

\item \label{enu:D8}
\[
\left\langle \delta_{x},\delta_{y}\right\rangle _{\mathscr{H}_{E}}=\begin{cases}
c\left(x\right)=\sum_{t\sim x}c_{xt} & \mbox{if \ensuremath{y=x}}\\
-c_{xy} & \mbox{if \ensuremath{\left(xy\right)\in E}}\\
0 & \mbox{if }\left(xy\right)\notin E,\; x\neq y
\end{cases}
\]

\end{enumerate}
\end{lem}
\begin{proof}
See \cite{JP10,JP11,MR2432048}. For the selfadjointness of the graph
Laplacian in $l^{2}\left(V\right)$, see Theorem \ref{thm:Dsa} below.\end{proof}
\begin{rem}
We will show in Section \ref{sec:bint} that the two $\infty\times\infty$
matrices 
\begin{eqnarray}
\Delta_{xy} & := & \left\langle \delta_{x},\delta_{y}\right\rangle _{\mathscr{H}_{E}},\;(\mbox{see}\:\eqref{enu:D8});\;\mbox{and}\label{eq:m11}\\
K_{xy} & := & \left\langle v_{x},v_{y}\right\rangle _{\mathscr{H}_{E}}\label{eq:m12}
\end{eqnarray}
are formal inverses; more precisely, for any $x,y\in V$, the following
$\infty\times\infty$ matrix-product, $\Delta K$ and $K\Delta$ are
well defined; and 
\begin{eqnarray}
\sum_{z\in V'}\Delta_{xz}K_{zy} & = & \delta_{x,y},\;\mbox{and}\label{eq:m13}\\
\sum_{z\in V'}K_{xz}\Delta_{zy} & = & \delta_{x,y}\label{eq:m14}
\end{eqnarray}
both hold. However, the operator theoretic interpretation of the two,
(\ref{eq:m13}) vs (\ref{eq:m14}), is different. \end{rem}
\begin{thm}[\cite{MR2432048,JP10,JP11,2007PhDT.......216W,Keller_2012}]
\label{thm:Dsa}Let $G=\left(E,V,c\right)$ be a weighted graph as
specified above; so with a given conductance function $c$ defined
on the set of edges $E$ of $G$; and let $\Delta$ be the corresponding
Laplace operator. Then, as an operator in $l^{2}(V)$ with domain
consisting of finitely supported functions, $\Delta$ is essentially
selfadjoint. \end{thm}
\begin{proof}
Below we give a new proof of this essential selfadjointness. One advantage
with the proof below is its use of different properties of the operator
$\Delta$ than was the case for earlier approaches. We also believe
that the idea used here has wider use; -- that it is applicable to
other operators in analysis and potential theory, both discrete and
continuous.

Note TFAE:
\begin{enumerate}[label=(\roman{enumi}),itemsep=5pt,ref=\roman{enumi}]
\item $f\in l^{2}\left(V\right)$ is a $\Delta$-defect vector; 
\item $\left\langle \varphi+\Delta\varphi,f\right\rangle _{l^{2}}=0$, $\forall\varphi\in span\left\{ \delta_{x}\right\} $;
\item $\left\langle \delta_{x}\left(\cdot\right)+\left(\Delta\delta_{x}\right)\left(\cdot\right),f\right\rangle _{l^{2}}=0$,
$\forall x\in V$;
\item $\big\langle\underset{\left(1+c\left(x\right)\right)\delta_{x}}{\underbrace{\delta_{x}+c\left(x\right)\delta_{x}}}-\sum_{y\sim x}c_{xy}\delta_{y},f\big\rangle_{l^{2}}=0$,
$\forall x\in V$; 
\item $\left(1+c\left(x\right)\right)f\left(x\right)-\sum_{y\sim x}c_{xy}f\left(y\right)=0$,
$\forall x\in V$;
\item $\left(1+c\left(x\right)\right)f\left(x\right)-c\left(x\right)\left(\mathbb{P}f\right)\left(x\right)=0$,
$\forall x\in V$; where 
\[
p_{xy}=\frac{1}{c\left(x\right)}c_{xy}\quad\mbox{and}\quad\left(\mathbb{P}f\right)\left(x\right)=\sum_{y\sim x}p_{xy}f\left(y\right);
\]
 
\item \label{enu:def7}$\left(\mathbb{P}f\right)\left(x\right)=\left(1+\frac{1}{c\left(x\right)}\right)f\left(x\right)$,
$\forall x\in V$. 
\end{enumerate}

With the splitting $f=\Re\left\{ f\right\} +i\Im\left\{ f\right\} $,
it is enough to consider the case when $f$ is real valued.

Since $f\in l^{2}\left(V\right)$, it has a local max, i.e., $\exists x_{0}\in V$
s.t. $f\left(\cdot\right)\leq f\left(x_{0}\right)$ in $V$. Assume
$f\left(x_{0}\right)>0$ (otherwise replace $f$ by $-f$.) Now, if
$f$ is a defect vector, we have 
\[
\left(1+\frac{1}{c\left(x_{0}\right)}\right)f\left(x_{0}\right)\underset{\left(\text{by }\left(\ref{enu:def7}\right)\right)}{=}\left(\mathbb{P}f\right)\left(x_{0}\right)\leq f\left(x_{0}\right)\Longrightarrow\frac{1}{c\left(x_{0}\right)}f\left(x_{0}\right)\leq0,
\]
which contradicts the assumption that $f\left(x_{0}\right)>0$. 

\end{proof}
\begin{thm}
\label{thm:esa}Let $\left(V,E,c,\Delta,\mathscr{H}_{E}\right)$ be
as above, and fix a base-point $o\in V$. Set $V':=V\backslash\left\{ o\right\} $.
Fix dipole $v_{x}:=v_{x,o}$, $x\in V'$ s.t. 
\begin{equation}
f\left(x\right)-f\left(o\right)=\left\langle v_{x},f\right\rangle _{\mathscr{H}_{E}},\quad\forall f\in\mathscr{H}_{E},\:\forall x\in V'.\label{eq:dp1}
\end{equation}
Set 
\begin{equation}
\left(\Delta^{-1}\right)_{xy}:=\left\langle v_{x},v_{y}\right\rangle _{\mathscr{H}_{E}},\quad\left(x,y\right)\in V'\times V'.\label{eq:dp2}
\end{equation}
Then $\Delta$ is \uline{not} essentially selfadjoint on $\mathscr{D}_{E}:=span\left\{ v_{x}\:\big|\: x\in V'\right\} $
if and only if there is a non-zero function $f\in\mathscr{H}_{E}$
such that
\begin{equation}
h\left(x\right):=f\left(x\right)+\sum_{y\in V'}\left(\Delta^{-1}\right)_{xy}f\left(y\right)\label{eq:dp3}
\end{equation}
is harmonic. \end{thm}
\begin{proof}
By general operator theory (see \cite{DS88b}), the essential selfadjointness
assertion holds iff the following implication holds:
\begin{equation}
\left[f\in\mathscr{H}_{E},\:\mbox{and}\:\left\langle \varphi+\Delta\varphi,f\right\rangle _{\mathscr{H}_{E}}=0,\;\forall\varphi\in\mathscr{D}_{E}\right]\Longrightarrow\big[f=0\big].\label{eq:dp4}
\end{equation}
Taking $\varphi=v_{x}$, and modulo an additive constant, we see that
a possible solution $f\in\mathscr{H}_{E}$ to (\ref{eq:dp4}) will
have the form, setting: 
\begin{equation}
\left(\mathbb{P}f\right)\left(x\right)=\left(1+\frac{1}{c\left(x\right)}\right)f\left(x\right),\quad\forall x\in V',\label{eq:dp5}
\end{equation}
where $\left(\mathbb{P}f\right)\left(x\right)=\sum_{y\sim x}p_{x,y}f\left(y\right)$,
$p_{x,y}=\dfrac{c_{x,y}}{c\left(x\right)}$. 

An iteration of (\ref{eq:dp5}) yields 
\begin{equation}
\left(\mathbb{P}^{n+1}f\right)\left(x\right)=f\left(x\right)+\sum_{k=0}^{n}\mathbb{P}^{k}\left(\frac{f}{c}\right)\left(x\right).\label{eq:dp6}
\end{equation}
But we have pointwise convergence on the RHS in (\ref{eq:dp6}), and
\[
\left(1-\mathbb{P}\right)^{-1}=\left(\frac{1}{c}\Delta\right)^{-1},\;\mbox{so}
\]
\begin{align}
\left(1-\mathbb{P}\right)^{-1}\left(\frac{f}{c}\right)\left(x\right) & =\Delta{}^{-1}\left(\text{diag}\left(c\right)\right)\left(\frac{f}{c}\right)\left(x\right)\nonumber \\
 & =\left(\Delta^{-1}f\right)\left(x\right)=\sum_{y}\left(\Delta^{-1}\right)_{xy}f\left(y\right).\label{eq:dp7}
\end{align}
So the LHS in (\ref{eq:dp6}) must converge pointwise; but it is clear
that $h=\lim_{n}\mathbb{P}^{n}f$ is harmonic. 

Finally, it is clear that every solution $f\in\mathscr{H}_{E}$ to
(\ref{eq:dp3}) will satisfy (\ref{eq:dp4}); which in turn is the
equation which decides non-essential selfadjointness, by general theory. \end{proof}
\begin{rem}
We introduce the Markov measure $\mu^{\left(Markov\right)}$ on the
space $\Omega$ of all $G=\left(V,E\right)$-paths, and the Markov-walk
process
\begin{equation}
\pi_{n}\left(\omega\right):=\omega_{n},\quad\forall\omega\in\Omega,\: n\in\mathbb{N}_{0},\label{eq:mk1}
\end{equation}
where $\omega=\left(\omega_{0},\omega_{1},\omega_{2},\cdots\right)$,
$\omega_{j}\in V$, $\left(\omega_{j}\omega_{j+1}\right)\in E$, $\forall j\in\mathbb{N}_{0}$.
Then the matrix product $\mathbb{P}^{k}$ in (\ref{eq:dp6}) is 
\begin{equation}
\mbox{Prob}\left(\left\{ \pi_{m+k}=y\:\big|\:\pi_{m}=x\right\} \right)=(\mathbb{P}^{k})_{x,y}.\label{eq:mk2}
\end{equation}
We shall return to this Markov-process in Section \ref{sec:path}
below.
\end{rem}

\section{From conductance to current flow}

Let $G=\left(V,E,c\right)$ be an infinite weighted graph (connected,
see (\ref{enu:net4}) before Definition \ref{def:cond}). Here, $V=$
vertex-set, $E=$ edges, and $c:E\rightarrow\mathbb{R}_{+}$ is a
fixed conductance function, so that $c=\left(c_{xy}\right)$, $\left(xy\right)\in E$.
Let $\mathscr{H}_{E}$ be the corresponding energy-Hilbert space (see
(\ref{eq:Einn})-(\ref{eq:Enorm})).

Set the current flow $I_{\left(xy\right)}:=\partial w$, where 
\begin{align}
I_{\left(xy\right)} & =\left(\partial w\right)\left(x,y\right)=c_{xy}\left(w\left(x\right)-w\left(y\right)\right),\quad\forall\left(xy\right)\in E,\; w\in\mathscr{H}_{E}.\label{eq:cr0}
\end{align}
And set
\begin{align}
Dissp & =\left\{ \partial w\;\big|\; w\in\mathscr{H}_{E},\;\left\Vert \partial w\right\Vert _{Diss}^{2}:=\sum\sum\frac{1}{c_{xy}}I_{\left(xy\right)}^{2}<\infty\right\} \label{eq:cr1}
\end{align}
as a weighted $l^{2}$-space on $E$, where $1/c_{xy}=$ resistance. 

On edges $\left(u,v\right)\in E$ from $x$ to $y$ in $V$, the current
$I_{\left(u,v\right)}$ is
\[
I_{uv}=c_{uv}\left(f\left(u\right)-f\left(v\right)\right)
\]
where $f$ denotes a voltage-distribution. See Fig \ref{fig:current}.

\begin{figure}[H]
\includegraphics[width=0.6\textwidth]{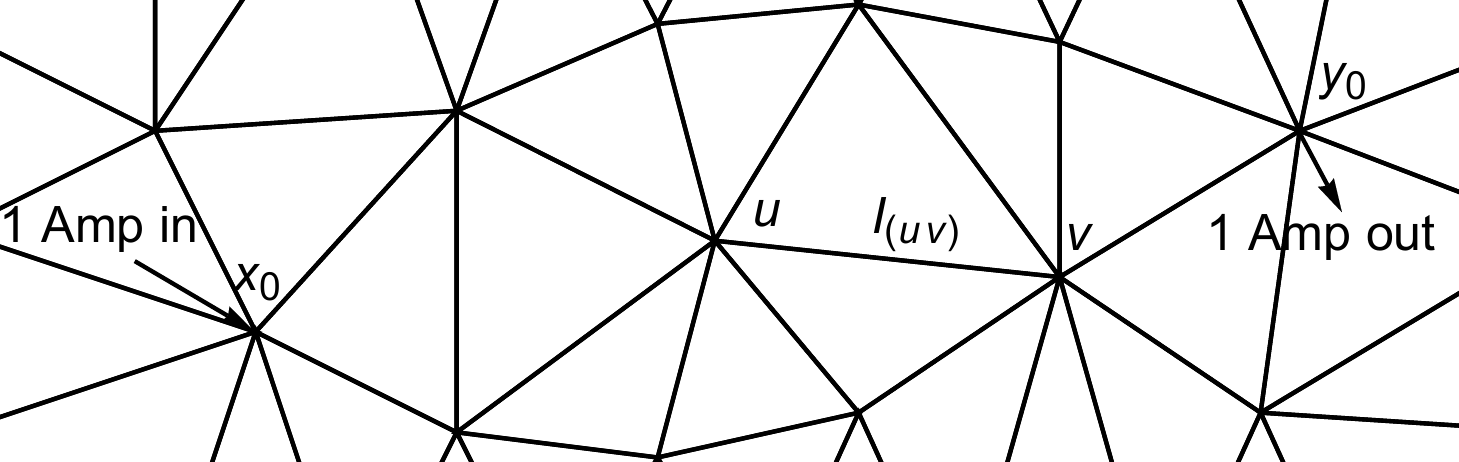}

\protect\caption{\label{fig:current}Current flows from vertex $x_{0}$ to vertex $y_{0}$;
with a given conductance function $c$. The convex set $W_{x_{0}y_{0}}$.}
\end{figure}

\begin{lem}
\label{lem:current}The operator $\partial:\mathscr{H}_{E}\rightarrow Dissp$
is isometric.\end{lem}
\begin{proof}
One checks that 
\begin{align}
\left\Vert w\right\Vert _{\mathscr{H}_{E}}^{2} & =\sum\sum c_{xy}\left|w\left(x\right)-w\left(y\right)\right|^{2}\quad(\text{energy})\label{eq:f1a}\\
 & =\sum\sum\frac{1}{c_{xy}}I_{\left(xy\right)}^{2}\quad(\text{dissipation})\label{eq:f12}
\end{align}
where $I_{xy}=\left(\partial w\right)_{xy}=c_{xy}\left(w\left(x\right)-w\left(y\right)\right)$,
and $1/c_{xy}=$ resistance on the edge $\left(xy\right)$, see (\ref{eq:cr0});
and the lemma follows.\end{proof}
\begin{defn}
Set $dist\left(x_{0},y_{0}\right)=$ distance $x_{0}\rightarrow y_{0}$
= voltage drop from $x_{0}$ to $y_{0}$ when current $I$ satisfies
$I=1$ at $x_{0}$ ``in'' and current $I=-1$ at $y_{0}$ ``out.''\end{defn}
\begin{thm}
There is a unique current flow s.t. 
\begin{align}
dist\left(x_{0},y_{0}\right) & =\inf\left\{ \left\Vert I\right\Vert _{Diss}^{2},\; I\big|_{\left\{ x_{0},y_{0}\right\} }=\text{1 Amp in, and 1 out}\right\} \label{eq:cr2}
\end{align}
\end{thm}
\begin{proof}
Recall that by Lemma \ref{lem:dipole}, $\exists!v_{xy}$ s.t. 
\begin{align}
\left\langle v_{xy},f\right\rangle _{\mathscr{H}_{E}} & =f\left(x\right)-f\left(y\right),\quad\forall\left(x,y\right)\in V\times V,\;\forall f\in\mathscr{H}_{E}.\label{eq:cr3}
\end{align}
Set $I=\partial v_{xy}$, then 
\begin{align}
d\left(x_{0},y_{0}\right)=\inf\left\Vert I\right\Vert _{Diss}^{2} & =\left\Vert \partial v_{x_{0}y_{0}}\right\Vert _{Diss}^{2}\nonumber \\
 & =\left\Vert v_{x_{0}y_{0}}\right\Vert _{\mathscr{H}_{E}}^{2}\left(=\text{resistance distance}\right);\label{eq:cr4}
\end{align}
i.e., the infimum in (\ref{eq:cr2}) is obtained at the flow $I=\partial v_{x_{0}y_{0}}$,
see (\ref{eq:cr3})-(\ref{eq:cr4}). For a proof, see \cite{JP10,JP11}.

The infimum in (\ref{eq:cr2}) and (\ref{eq:cr4}) is justified with
the following Hilbert space geometry applied to the energy-Hilbert
space $\mathscr{H}_{E}$: 

The infimum in (\ref{eq:cr2}) is attained when $I_{0}=\partial v_{x_{0}y_{0}}$.
We use that $I_{0}$ is the vector in the convex set $W_{x_{0}y_{0}}$
of minimum norm. Since $\partial$ from Lemma \ref{lem:current} is
isometric, we see that $W_{x_{0}y_{0}}$ is both closed and convex.
From Hilbert space geometry, see e.g. \cite{MR1157815}, we know that
$W_{x_{0}y_{0}}$ contains a vector of smallest norm. From the definition
of $W_{x_{0}y_{0}}$ (see e.g., Fig \ref{fig:current}), we conclude
that the minimum must be $I_{0}=\partial v_{x_{0}y_{0}}$; see also
\cite{JP11}.\end{proof}
\begin{rem}
The function $v_{x_{0}y_{0}}$ in (\ref{eq:cr4}) is called a \emph{dipole},
and it satisfies 
\begin{equation}
\Delta v_{x_{0}y_{0}}=\delta_{x_{0}}-\delta_{y_{0}}
\end{equation}
where $\Delta$ is the Laplacian from (\ref{eq:lap}).
\end{rem}
Below, we offer seven different, but equivalent, formulas for the
resistance metric $d_{res}\left(x,y\right)$:
\begin{thm}[\cite{JP11}]
 \label{thm:metric}Let $V,E,c,\Delta$, and $d_{res}$ be as above;
let $x,y\in V$, and let $W_{xy}$ be the set from Fig \ref{fig:current}.
Then
\begin{eqnarray*}
d_{res}\left(x,y\right) & = & v\left(x\right)-v\left(y\right)\;\mbox{when}\; v=v_{x,y}\\
 & = & \left\Vert v_{x,y}\right\Vert _{\mathscr{H}_{E}}^{2}\\
 & = & \min\left\{ \left\Vert I\right\Vert _{Dissp}^{2}\::\: I\in W_{x,y}\right\} \\
 & = & \left\Vert w\right\Vert _{\mathscr{H}_{E}}^{2}\;\mbox{when}\;\Delta w=\delta_{x}-\delta_{y}\\
 & = & \frac{1}{\min\left\{ \left\Vert w\right\Vert _{\mathscr{H}_{E}}^{2}\::\: w\in\mathscr{H}_{E},\:\left|w\left(x\right)-w\left(y\right)\right|=1\right\} }\\
 & = & \inf\left\{ K\::\:\left|w\left(x\right)-w\left(y\right)\right|^{2}\leq K\left\Vert w\right\Vert _{\mathscr{H}_{E}}^{2},\: w\in\mathscr{H}_{E}\right\} \\
 & = & \sup\left\{ \left|w\left(x\right)-w\left(y\right)\right|^{2}\::\: w\in\mathscr{H}_{E},\:\left\Vert w\right\Vert _{\mathscr{H}_{E}}\leq1\right\} .
\end{eqnarray*}
\end{thm}
\begin{example}[see Fig \ref{fig:res}]
\[
d_{res}\left(x,y\right)=r_{1}+\frac{1}{\frac{1}{r_{2}}+\frac{1}{r_{3}}}
\]

\end{example}
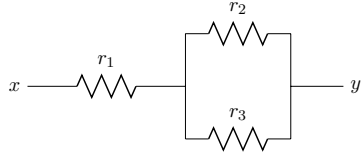
\begin{figure}
\scalebox{0.7}{
\begin{circuitikz}
\draw 
node[left]{$x$} (0,0) to [R=$r_{1}$] (3,0)
(3,0) to (3,1) to [R=$r_{2}$] (5,1) to (5,-1)

(3,0) to (3,-1) to [R=$r_{3}$] (5,-1)

(5,0) to (6,0) node[right]{$y$}
;
\end{circuitikz}
}

\protect\caption{\label{fig:res}Example of resistor configuration in a network: configuration
of three resistors, having values $r_{1},r_{2},r_{3}$ Ohm.}

\end{figure}

\section{The metric boundary}
\begin{defn}
By $M$ we mean the set of equivalence classes of sequences $\left(x_{i}\right)\subset V$
of vertices s.t. 
\begin{equation}
\lim_{i,j\rightarrow\infty}d\left(x_{i},x_{j}\right)=0\;\text{(Cauchy); and}\label{eq:md1}
\end{equation}
\begin{equation}
\left(x_{i}\right)\sim\left(y_{i}\right)\underset{(\text{Def})}{\Longleftrightarrow}\lim_{i\rightarrow\infty}d\left(x_{i},y_{i}\right)=0.\label{eq:md2}
\end{equation}

\end{defn}
The vertex-set $V$ is identified with a subset of $M$ via the mapping
$\gamma:V\rightarrow M$, 
\begin{equation}
V\ni x\longmapsto\gamma\left(x\right)=\mbox{class}\left(x,x,x,\cdots\right).\label{eq:md3}
\end{equation}
Hence $b\in M\backslash V\text{(=bdd \ensuremath{V})}$ iff $b=\left(y_{i}\right)\in M$
satisfies the following: 

\begin{equation}
\forall x\in V,\:\exists\varepsilon\in\mathbb{R}_{+},\:\exists\left(y_{i_{k}}\right)\subset\left(y_{i}\right),\;\mbox{s.t.}\; d\left(x,y_{i_{k}}\right)\geq\varepsilon,\quad\forall k\in\mathbb{N}.\label{eq:md4}
\end{equation}
 Note that the assertion in (\ref{eq:md4}) states that: 
\begin{equation}
d\left(\gamma\left(x\right),b\right)>0,\quad\forall x\in V.\label{eq:md5}
\end{equation}

We now show that if $d:=d_{res}$ is \emph{bounded}, then every function
$f\in\mathscr{H}_{E}$ extends by closure to $M$: If $b\in M$, and
$\left\{ x_{i}\right\} \subset V$, are such that $\lim_{i\rightarrow\infty}d\left(x_{i},b\right)=0$,
we set 
\begin{equation}
\widetilde{f}\left(b\right)=\lim_{i\rightarrow\infty}f\left(x_{i}\right).\label{eq:m1}
\end{equation}
It is then immediate that 
\begin{equation}
\big|\widetilde{f}\left(b\right)-\widetilde{f}\left(b'\right)\big|^{2}\leq d\left(b,b'\right)\left\Vert f\right\Vert _{\mathscr{H}_{E}}^{2}.\label{eq:m2}
\end{equation}
 
\begin{thm}
\label{thm:algebra}If the resistance metric $d=d_{res}$ is bounded
on $V\times V$, then 
\begin{align}
\mathscr{H}_{E} & \subset l^{\infty}\left(V\right),\;\mbox{and}\quad\widetilde{\mathscr{H}}_{E}\subseteq C\left(M\right);\label{eq:gg45}
\end{align}
i.e., every energy function $w$ on $V$ is bounded, and $\mathscr{H}_{E}$
is an algebra under pointwise product. \end{thm}
\begin{proof}
The containment in (\ref{eq:gg45}) follows from the estimate (\ref{eq:gg3}). 

We proceed to show that $\mathscr{H}_{E}$ is an algebra when $\left(V,d\right)$
is assumed bounded:

Let $u,w\in\mathscr{H}_{E}$, then 
\begin{align*}
\left(uw\right)\left(x\right) & :=u\left(x\right)w\left(x\right),\quad\forall x\in V;
\end{align*}
satisfies 
\begin{align}
\left\Vert uw\right\Vert _{\mathscr{H}_{E}}^{2} & \leq\left(\left\Vert u\right\Vert _{\infty}^{2}+\left\Vert w\right\Vert _{\infty}^{2}\right)\left(\left\Vert u\right\Vert _{\mathscr{H}_{E}}^{2}+\left\Vert w\right\Vert _{\mathscr{H}_{E}}^{2}\right).\label{eq:ff1}
\end{align}
Since $u,w\in l^{\infty}\left(V\right)$, it follows that $uw\in\mathscr{H}_{E}$,
i.e., $\left\Vert uw\right\Vert _{\mathscr{H}_{E}}<\infty$. 

The proof of (\ref{eq:ff1}) is as follows:
\begin{eqnarray*}
 &  & \sum_{E}c_{xy}\left|\left(uw\right)\left(x\right)-\left(uw\right)\left(y\right)\right|^{2}\\
 & = & \sum_{E}c_{xy}\left|u\left(x\right)\left(w\left(x\right)-w\left(y\right)\right)+w\left(y\right)\left(u\left(x\right)-u\left(y\right)\right)\right|^{2}\\
 & \underset{\left(\text{Schwarz}\right)}{\leq} & \sum_{E}c_{xy}\left(\left|u\left(x\right)\right|^{2}+\left|w\left(y\right)\right|^{2}\right)\left(\left|u\left(x\right)-u\left(y\right)\right|^{2}+\left|w\left(x\right)-w\left(y\right)\right|^{2}\right)\\
 & \leq & \left(\left\Vert u\right\Vert _{\infty}^{2}+\left\Vert w\right\Vert _{\infty}^{2}\right)\left(\sum_{E}c_{xy}\left|u\left(x\right)-u\left(y\right)\right|^{2}+\sum_{E}c_{xy}\left|w\left(x\right)-w\left(y\right)\right|^{2}\right)
\end{eqnarray*}
which is the desired estimate.\end{proof}
\begin{rem}
After the completion of a first version of our paper, D. Lenz kindly
informed us that a version of Theorem \ref{thm:algebra} above is
in the paper \cite{Georgakopoulos2014}. Our approach and aim here
is different though.\end{rem}
\begin{cor}
\label{cor:fext}Let $V,E,c,d=d_{res}$ be as above, i.e., assume
that $d$ is bounded, and that $M$ is compact; then when the constant
function $\mathbbm{1}$ on $M$ is adjoined 
\begin{equation}
\widetilde{\mathscr{H}}_{E}=\{\widetilde{f}\:|\: f\in\mathscr{H}_{E}\}\subset C\left(M\right)\label{eq:m3}
\end{equation}
is a dense subalgebra; dense in the uniform norm on $C\left(M\right)$.\end{cor}
\begin{proof}
We already proved that $\widetilde{\mathscr{H}}_{E}$ is an algebra
of continuous functions on $M$ (= the metric completion of $\left(V,d_{res}\right)$),
so we only need to show that it is dense in the $\left\Vert \cdot\right\Vert _{\infty}$-norm
on $M$. Since $M$ is compact, $\Vert\widetilde{f}\Vert_{\infty}=\max\{|\widetilde{f}\left(b\right)|\::\: b\in M\}$. 

It is clear that $\widetilde{\mathscr{H}}_{E}$ is closed under complex
conjugation; so, by the Stone-Weierstrass theorem, we only need to
prove that it separates points. We will prove that if $b\neq b'$
in $M$ then there is a vertex $x\in V$ such that $\widetilde{v}_{x}\left(b\right)\neq\widetilde{v}_{x}\left(b'\right)$.

Since $M$ is the metric completion of $\left(V,d\right)$, it is
enough to show that $\widetilde{\mathscr{H}}_{E}$ separate points
in $V$. Assume the contrary: that there are vertices $y,z\in V$,
$y\neq z$ such that $v_{x}\left(y\right)=v_{x}\left(z\right)$ holds
for all $x\in V$; in other words, $\left\langle v_{x},v_{y}-v_{z}\right\rangle _{\mathscr{H}_{E}}=0$
holds for all $x\in V$. But $span\left\{ v_{x}\:|\: x\in V\right\} $
is dense in $\mathscr{H}_{E}$; and so $v_{y}-v_{z}=0$, contradicting
$d\left(y,z\right)=\left\Vert v_{y}-v_{z}\right\Vert _{\mathscr{H}_{E}}^{2}>0$. 
\end{proof}

\section{\label{sub:mcom}Discrete resistance metric -- metric completions}

Set $d:=d_{res}$ the resistance metric, see (\ref{eq:cr4}). Let
$(M,\widetilde{d})$ be the metric completion of $\left(V,d\right)$,
i.e., $V$ consists of a metric space $M$ with the metric
\begin{align}
d_{res}\left(x,y\right) & =\left\Vert v_{xy}\right\Vert _{\mathscr{H}_{E}}^{2}\label{eq:d1}
\end{align}
where $v_{xy}$ is the \emph{dipole} vector in (\ref{eq:dipole}). 

It is an important theorem \cite{JP11,JP10} that $d_{res}$ in (\ref{eq:d1})
is indeed a metric, i.e., that 
\begin{align}
d_{res}\left(x,y\right) & \leq d_{res}\left(x,z\right)+d_{res}\left(z,y\right)\label{eq:d2}
\end{align}
holds for $\forall x,y,z\in V$. This result applies to all weighted
graph models where $d_{res}$ is computed from a fixed conductance
function $c:E\rightarrow\mathbb{R}_{+}$.

Let 
\begin{align}
K\left(x,y\right) & =\left\langle v_{x},v_{y}\right\rangle _{\mathscr{H}_{E}},\label{eq:g2a}
\end{align}
then the triangle inequality (\ref{eq:d2}) is equivalent to 
\begin{align}
K\left(x,y\right) & \geq K\left(x,z\right)+K\left(z,y\right)-K\left(z,z\right),\quad\forall x,y,z\in V.\label{eq:gg1}
\end{align}

Now assume that $d=d_{res}$ is \emph{bounded}. 
\begin{defn}
\label{def:typeA}We say that a system $(V,E,c,d_{res})$ is type
$A$ if whenever $\lim_{j}v_{x_{j}}$ exists in $C\left(V,d\right)$
then $\left(x_{j}\right)$ is a Cauchy sequence in $(V,d)$.\end{defn}
\begin{thm}
\label{thm:Mcpt}If $d_{res}$ is bounded on $V\times V$, and assume
the system $\left(V,E,c,d_{res}\right)$ is of type $A$; then $(M,\widetilde{d})$
is a compact metric space.\end{thm}
\begin{proof}
Fix a base-point $o\in V$, and set $v_{x}=v_{x,o}$, $x\in V\backslash\left\{ o\right\} $,
then 
\begin{align}
v_{xy} & =v_{x}-v_{y}\label{eq:gg2}
\end{align}
as follows from (\ref{eq:gg1}); also see Lemma \ref{lem:dipole}.
By Schwarz, applied to the energy Hilbert space $(\mathscr{H}_{E},\left\langle \cdot,\cdot\right\rangle _{\mathscr{H}_{E}})$,
we get the following Lipschitz-estimate:
\begin{align}
\left|f\left(x\right)-f\left(y\right)\right|^{2} & \leq d\left(x,y\right)\left\Vert f\right\Vert _{\mathscr{H}}^{2},\quad\forall f\in\mathscr{H}_{E},\; x,y\in V.\label{eq:gg3}
\end{align}
Consequences of (\ref{eq:gg1})-(\ref{eq:gg3}): 
\begin{enumerate}
\item Every $f\in\mathscr{H}_{E}$ extends to a uniformly continuous function
$\widetilde{f}$ on $M$; extension by metric limits. 
\item If $x_{i}\in V$, and $d\left(x_{i},x_{j}\right)\rightarrow0$, for
$i,j\rightarrow\infty$, then $f\left(x_{i}\right)$ has a limit in
$\mathbb{C}$ (or $\mathbb{R}$). Set 
\begin{align}
\widetilde{x} & \in M,\quad\widetilde{x}=\lim_{i}x_{i}.\label{eq:gg4}
\end{align}
If $\left(x_{i}\right),\left(y_{i}\right)\subset V$ are Cauchy sequences,
set (the extended metric $\widetilde{d}$):
\begin{align}
\widetilde{d}\left(\widetilde{x},\widetilde{y}\right) & =\lim_{i\rightarrow\infty}d\left(x_{i},x_{j}\right);\label{eq:gg5}
\end{align}
then by (\ref{eq:gg3}), we get 
\begin{align}
\big|\widetilde{f}\left(\widetilde{x}\right)-\widetilde{f}\left(\widetilde{y}\right)\big|^{2} & \leq\widetilde{d}\left(\widetilde{x},\widetilde{y}\right)\left\Vert f\right\Vert _{\mathscr{H}_{E}.}^{2}\label{eq:gg6}
\end{align}

\end{enumerate}

The assertion in the theorem follows from the considerations below. 

\end{proof}
\begin{lem}
An application of Arzelà\textendash Ascoli shows that 
\begin{align}
\left\{ \widetilde{f}\subset C\left(M\right)\:\big|\: f\in\mathscr{H}_{E},\left\Vert f\right\Vert _{\mathscr{H}_{E}}\leq1\right\} \label{eq:gg7}
\end{align}
is relatively compact in $C\left(M\right)$, in the Montel topology
of uniform convergence on compact sets.
\end{lem}
But if $d$ is bounded on $V\times V$, then 
\begin{align}
\left\Vert v_{x_{i}}\right\Vert _{\mathscr{H}_{E}} & \leq A\cdot K_{d}\label{eq:gg8}
\end{align}
where $A$ is a fixed global constant, since 
\begin{align}
d\left(x_{i},x_{j}\right) & =\left\Vert v_{x_{i}}-v_{x_{j}}\right\Vert _{\mathscr{H}_{E}}^{2}.\label{eq:gg9}
\end{align}
Hence by (\ref{eq:gg7}) with $K_{d}$ in place of 1, we get that:
\begin{cor}
\label{cor:Mcpt}Assume type $A$, then for every sequence $x_{1},x_{2},x_{3},\cdots$
in $V$, $\exists$ subsequence $\left(x_{i_{k}}\right)$ s.t.
\begin{enumerate}[label=(\roman{enumi}),ref=\roman{enumi}]
\item $\lim_{k\rightarrow\infty}x_{i_{k}}=b\in M$; and
\item Let $f_{lim}\in C\left(M\right)$ be the limit of the subsequence
$\left\{ \widetilde{v}_{x_{i}}\right\} \subset C\left(M\right)$,
then 
\begin{align*}
\lim_{k\rightarrow\infty}\widetilde{v}_{x_{i_{k}}}\left(b\right) & =f_{lim}\left(b\right).
\end{align*}

\end{enumerate}
\end{cor}
\begin{proof}
To see that $b\in M$, note that 
\begin{align*}
d\left(x_{i_{k}},x_{i_{l}}\right) & =\Vert v_{x_{i_{k}}}-v_{x_{i_{l}}}\Vert_{\mathscr{H}_{E}}^{2}=\left|\widetilde{v}_{i_{k}i_{l}}\left(x_{i_{k}}\right)-\widetilde{v}_{i_{k}i_{l}}\left(x_{i_{l}}\right)\right|\xrightarrow[\; k,l\rightarrow\infty\;]{}0;
\end{align*}
since by (\ref{eq:gg6}), the functions $\widetilde{v}_{i_{k}i_{l}}\left(\cdot\right)$
are uniformly bounded, and equicontinuous on $M$. Since we assume
the system $\left(V,E,c,d_{res}\right)$ is of type $A$, it follows
that every sequence $x_{1},x_{2}\cdots$ in $V$ has a convergence
subsequence with limit in $M$. By the definition of $M$, the same
is true for $M$, and so $M$ is compact: Every sequence $b_{1},b_{2},\cdots\subset M$
contains a convergent subsequence. \end{proof}
\begin{rem}
The following example from \cite{Georgakopoulos2014} shows that our
assumed condition ``type $A$'' in Theorem \ref{thm:Mcpt} and Corollary
\ref{cor:Mcpt} cannot be omitted. There are bounded resistance metrics
(non-type $A$) for which the corresponding completions are non-compact.

We learned from D. Lenz that the boundedness of the resistance metric
does not imply the completion $(M,\widetilde{d})$ is compact \cite{Georgakopoulos2014}.
Indeed, the type $A$ assumption for the system $\left(V,E,c,d_{res}\right)$
is required. (See Definition \ref{def:typeA}.)\end{rem}
\begin{example}[Example 8.6 in \cite{Georgakopoulos2014}]
\label{exa:cpt}Fig \ref{fig:2tra} below is a tree-like graph with
many ends all of which have bounded distance to the root (making the
resistance metric bounded) but at the same time being too far apart
from each other to be covered by finitely many balls of an fixed but
arbitrarily small size. Thus, the weighted graph in this case is bounded
with respect to $d_{res}$ metric and the completion is not compact
with respect to the resistance metric. 

The graph basically consists of a copy of the natural numbers with
the property that each natural number has a ray emanating from it
and this ray being again the natural numbers. There are weights (Fig
\ref{fig:2trb}) on the graph making all these copies of the natural
numbers having bounded diameter in the resistance metric. This makes
the resistance metric on this graph bounded. On the other hand a point
far out in one of the emanating rays has a uniform distance to any
point far out in any other emanating ray. This makes the example non-totally
bounded. Hence, the example has the mentioned properties.
\end{example}
\begin{figure}
\includegraphics[width=0.5\textwidth]{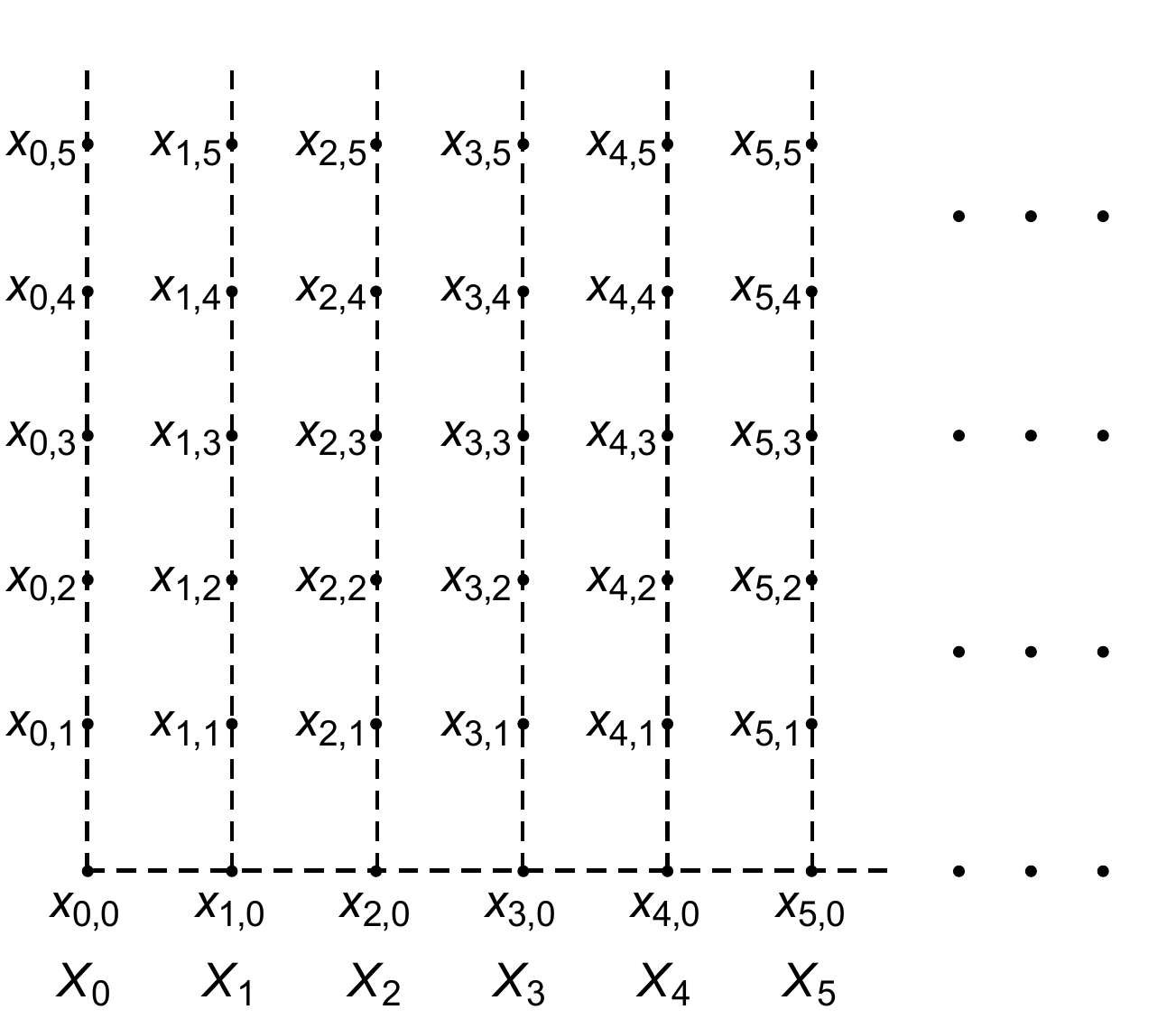}
\[
V=\bigcup_{n=0}^{\infty}X_{n},\quad X_{n}=\left\{ x_{nk}\::\: k=0,1,2,\cdots\right\} 
\]

\protect\caption{\label{fig:2tra}A double infinite planar graph: An infinitely long
comb as an infinite array of teeth, each tooth infinitely long. }

\end{figure}

\begin{figure}
\includegraphics[width=0.5\textwidth]{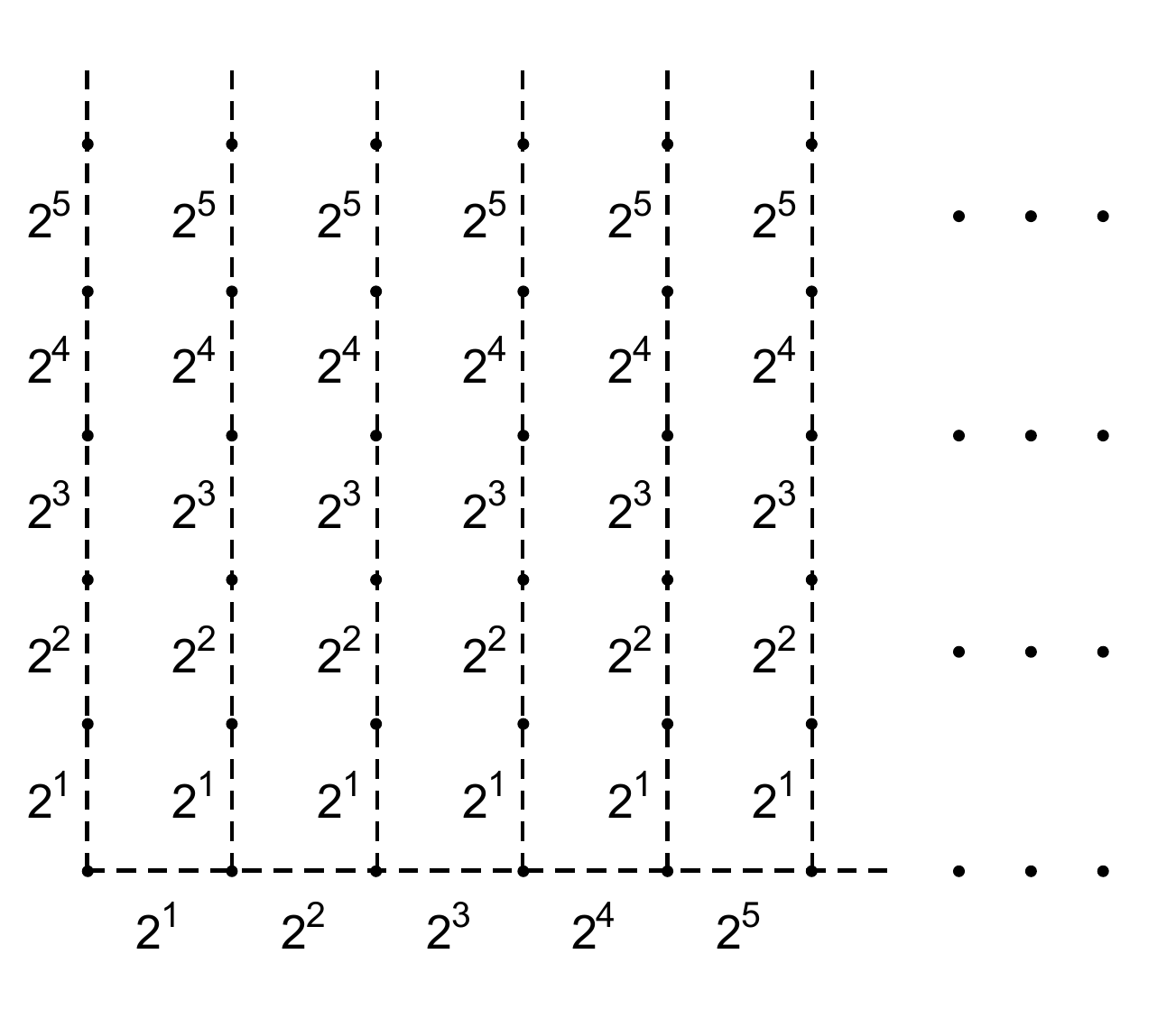}

\protect\caption{\label{fig:2trb}The conductance function $c_{\left(xy\right)}$. }
\end{figure}

\begin{lem}
Let $G=\left(V,E,c\right)$ be the weighted graph in Example \ref{exa:cpt}.
Fix a base-point $o\in V$, and set $\mathscr{D}_{E}=span\left\{ v_{x}\:\big|\: x\in V\backslash\left\{ o\right\} \right\} $
(see (\ref{eq:domE})). Then $\Delta\Big|_{\mathscr{D}_{E}}$, as
a densely defined Hermitian operator in the energy-Hilbert space $\mathscr{H}_{E}$,
is not essentially selfadjoint. Moreover, the deficiency indices are
$\left(\infty,\infty\right)$. \end{lem}
\begin{proof}
Let the $c$ be the conductance function as specified in Fig \ref{fig:2tra}-\ref{fig:2trb}.
Recall that 
\begin{align*}
\left(\Delta f\right)\left(x\right) & =\sum_{x\sim y}c_{xy}\left(f\left(x\right)-f\left(y\right)\right)=c\left(x\right)\left(f\left(x\right)-\sum_{x\sim y}p_{xy}f\left(y\right)\right)\\
 & =c\left(I-\mathbb{P}\right)f\left(x\right),\;\mbox{for all functions}\: f\:\mbox{on}\: V;
\end{align*}
where $c\left(x\right)=\sum_{y\sim x}c_{xy}$, $p_{xy}=c_{xy}/c\left(x\right)=$
transition probability, and $\left(\mathbb{P}f\right)\left(x\right)=\sum_{y\sim x}p_{xy}f\left(y\right)$.
Also see Theorem \ref{thm:esa}, Example \ref{exa:Zlat}, \ref{exa:Ntree},
and Remark \ref{rem:rwalk}. 

Suppose $f$ is a defect vector for $\Delta$. Since $\Delta$ is
positive, it suffices to consider $\Delta f=-f$. Note that 
\begin{equation}
\Delta f=-f\Longleftrightarrow c\left(I-\mathbb{P}\right)f=-f\Longleftrightarrow\mathbb{P}f=\left(1+\frac{1}{c}\right)f.\label{eq:df1}
\end{equation}
We proceed to show that $f$ is in $\mathscr{H}_{E}$, i.e., $\left\Vert f\right\Vert _{\mathscr{H}_{E}}<\infty$. 

Let $V=\left\{ x_{n,k}\right\} $ be the vertex-set as specified in
Fig \ref{fig:2tra}. Then, we have
\begin{eqnarray}
c\left(x_{n,k}\right) & = & 2^{k}+2^{k+1}\label{eq:df0}\\
p_{x_{n,k},x_{n,k-1}} & = & \frac{2^{k}}{2^{k}+2^{k+1}}=\frac{1}{3}\label{eq:df01}\\
p_{x_{n,k},x_{n,k+1}} & = & \frac{2^{k+1}}{2^{k}+2^{k+1}}=\frac{2}{3}\label{eq:df02}
\end{eqnarray}
and so 
\begin{eqnarray*}
\left(\mathbb{P}f\right)\left(x_{n,k}\right) & = & \frac{1}{3}f\left(x_{n,k-1}\right)+\frac{2}{3}f\left(x_{n,k+1}\right),\;\mbox{and}\\
\left(1+\frac{1}{c}\right)f\left(x_{n,k}\right) & = & \left(1+\frac{1}{2^{k}\cdot3}\right)f\left(x_{n,k}\right);\;\mbox{see}\:\text{\eqref{eq:df0}-\eqref{eq:df02}.}
\end{eqnarray*}
Thus, the defect vector $f$ satisfies $\Delta f=-f$ $\Longleftrightarrow$
\begin{equation}
\frac{1}{3}f\left(x_{n,k-1}\right)+\frac{2}{3}f\left(x_{n,k+1}\right)=\left(1+\frac{1}{2^{k}\cdot3}\right)f\left(x_{n,k}\right).\label{eq:df2}
\end{equation}
Set 
\begin{equation}
l_{k}:=l_{n,k}=f\left(x_{n,k}\right);\label{eq:df3}
\end{equation}
then we get the following recursive equation:
\begin{eqnarray}
\frac{1}{3}l_{k-1}+\frac{2}{3}l_{k+1} & = & \left(1+\frac{1}{2^{k}\cdot3}\right)l_{k};\label{eq:df4}
\end{eqnarray}
i.e., 
\[
l_{k+1}=\frac{3}{2}\left[\left(1+\frac{1}{2^{k}\cdot3}\right)l_{k}-\frac{1}{3}l_{k-1}\right]=\left(\frac{3}{2}+\frac{1}{2^{k+1}}\right)l_{k}-\frac{1}{2}l_{k-1}.
\]
Or, using matrix notation, we have 
\begin{equation}
\begin{pmatrix}l_{k+1}\\
l_{k}
\end{pmatrix}=\begin{pmatrix}\frac{3}{2}+\frac{1}{2^{k+1}} & -\frac{1}{2}\\
1 & 0
\end{pmatrix}\begin{pmatrix}l_{k}\\
l_{k-1}
\end{pmatrix}.\label{eq:df5}
\end{equation}

The asymptotic estimate of the sequence $\left(l_{k}\right)$ can
be derived from the eigenvalues of the coefficient matrix in (\ref{eq:df5}).
Note the eigenvalues are given by 
\[
x_{\pm}=\frac{\frac{3}{2}-\frac{1}{2^{k+1}}\pm\sqrt{\left(\frac{3}{2}-\frac{1}{2^{k+1}}\right)^{2}-2}}{2}\sim\frac{\frac{3}{2}\pm\frac{1}{2}}{2},\quad\text{asymptotically.}
\]

\emph{Conclusion}. The root $x_{-}=\frac{1}{2}$ shows that $l_{k}\sim1/2^{k}$
so $f\left(x_{n,k}\right)\sim1/2^{k}$ asymptotically. Consequently,
\begin{eqnarray*}
\left\Vert f\right\Vert _{\mathscr{H}_{E}}^{2} & \sim & \sum_{k}2^{k}\left(\frac{1}{2^{k}}-\frac{1}{2^{k+1}}\right)^{2}+\sum_{n}2^{n}\left(\frac{1}{2^{n}}-\frac{1}{2^{n+1}}\right)^{2}\\
 & \sim & \sum_{k}\frac{1}{2^{k}}+\sum_{n}\frac{1}{2^{n}}<\infty.
\end{eqnarray*}
Therefore, the corresponding defect vector $f$ is in $\mathscr{H}_{E}$,
and so $\Delta\big|_{\mathscr{H}_{E}}$ is not essentially selfadjoint.
\end{proof}
Set the Gelfand space $G_{E}=\left\{ \beta:\mathscr{H}_{E}\rightarrow\mathbb{C}(\text{or \ensuremath{\mathbb{R}}})\right\} $
(see \cite{MR1157815}) s.t. 
\begin{align}
\beta\left(uw\right) & =\beta\left(u\right)\beta\left(w\right),\quad\forall u,w\in\mathscr{H}_{E};\label{eq:gg46}
\end{align}
i.e., \emph{multiplicative functionals}.
\begin{defn}
Let $M:=$ metric completion of $\left(V,d_{res}\right)$. Set 
\begin{align*}
\left(x_{i}\right) & \sim\left(y_{i}\right)\overset{\text{Def}}{\Longleftrightarrow}d_{res}\left(x_{i},y_{i}\right)\rightarrow0
\end{align*}
for all Cauchy sequences $\left(x_{i}\right),\left(y_{i}\right)\subset V$. \end{defn}
\begin{thm}
\label{thm:gel}$M\subset G_{E}$, see (\ref{eq:gg46}). (The metric
completion is contained in the Gelfand space.)\end{thm}
\begin{proof}
Every $w\in\mathscr{H}_{E}$ extends by closure to $M$, by 
\begin{align}
\widetilde{w}\left(\widetilde{x}\right) & =\lim_{i\rightarrow\infty}w\left(x_{i}\right),\;\mbox{if}\; d_{res}\left(x_{i},x_{j}\right)\rightarrow0.\label{eq:gg47}
\end{align}
To see this, use the estimate $\left|w\left(x\right)-w\left(y\right)\right|^{2}\leq d\left(x,y\right)\left\Vert w\right\Vert _{\mathscr{H}_{E}}^{2}$,
$\forall w\in\mathscr{H}_{E}$; see (\ref{eq:gg3}).

Now, set $\beta_{\widetilde{x}}\left(w\right)=\widetilde{w}\left(\widetilde{x}\right)$,
and note that (\ref{eq:gg46}) is then immediate. (In fact, $M$ is
a compact metric space if $d_{res}$ is bounded.)\end{proof}
\begin{rem}
It was proved in \cite{Georgakopoulos2014} that the Gelfand space
is the Royden compactification; see \cite{Georgakopoulos2014} for
details. \end{rem}
\begin{question}
In these examples, what is the connection between (1) $M$, (2) $G_{E}$,
and (3) the $\infty$ path space $\Omega$ models? (Recall the $\omega\in\Omega$,
when $\omega=\left(x_{i}\right)_{i\in\mathbb{N}}$, where $x_{i}\in V$,
$\left(x_{i}x_{i+1}\right)\in E$.) \end{question}
\begin{rem}
In some cases $\left(x_{i}\right)_{i\in\mathbb{Z}}$, $\left(x_{i}x_{i+1}\right)\in E$
s.t. 
\begin{equation}
d_{res}\left(x_{i},x_{j}\right)\xrightarrow[\; i,j\rightarrow\infty\;]{}0,\label{eq:q8}
\end{equation}
there is a mapping $\psi:\Omega\rightarrow M\rightarrow G_{E}$. Note
that (\ref{eq:q8}) holds if $d_{res}$ is bounded. In this case $\exists$
continuous extension $\Omega\xrightarrow{\;\psi\;}M$, with $\psi\left(\omega\right)=\widetilde{x}$
where $\omega=\left(x_{i}\right)_{i\in\mathbb{Z}}\in\Omega$, and
where $\widetilde{x}\in M$ is the metric limit $\widetilde{d}_{res}\left(x_{i},\widetilde{x}\right)\rightarrow0$
if $d\left(x_{i},x_{j}\right)\rightarrow0$. See Fig \ref{fig:path}.
\end{rem}
\begin{figure}
\includegraphics[width=0.3\textwidth]{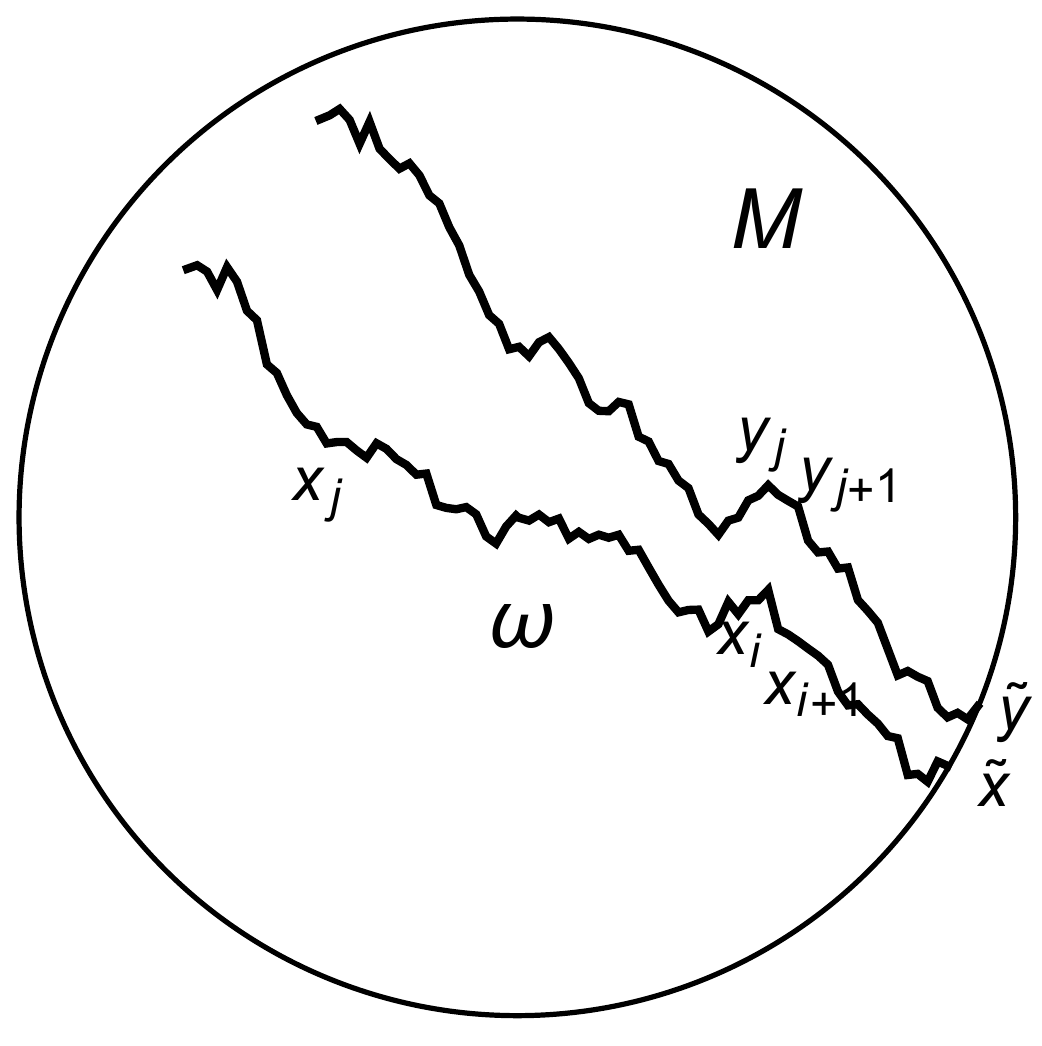}

\protect\caption{\label{fig:path}$\psi\left(\omega\right)=\widetilde{x}$; $\Omega=$
infinite path-space; $M=$ metric completion of $\left(V,d_{res}\right)$. }

\end{figure}

Actually even if not every path $\omega=\left(x_{i}\right)_{i\in\mathbb{Z}}$
satisfies $d_{res}\left(x_{i},x_{j}\right)\rightarrow0$, we can pass
to a sub-sequence.
\begin{thm}
\label{thm:cpt1}Assume that $d_{res}$ is type $A$ and bounded on
$V\times V$ (thus $(M,\widetilde{d}_{res})$ is compact by Theorem
\ref{thm:Mcpt}), and that $\omega=\left(x_{i}\right)_{i\in\mathbb{Z}}\in\Omega$,
then $\exists$ subsequence $\left\{ x_{i_{1}},x_{i_{2}},\cdots\right\} \subset\omega$,
and $\exists\widetilde{x}\in M$ s.t. 
\[
\widetilde{d}_{res}\left(x_{i_{k}},\widetilde{x}\right)\xrightarrow[\; k\rightarrow\infty\;]{}0.
\]
\end{thm}
\begin{proof}
(Application of Arzelà\textendash Ascoli) Recall that $v_{i}:=v_{x_{i},o}$,
where 
\[
\left|v_{i}\left(z\right)\right|^{2}=\left|\left\langle v_{i},v_{z}\right\rangle \right|^{2}\leq d\left(i,o\right)d\left(z,o\right)\leq K;
\]
which implies that 
\[
\left|v_{i}\left(z\right)-v_{i}\left(z'\right)\right|^{2}\leq K\: d\left(z,z'\right).
\]
By Arzelà\textendash Ascoli, $\exists$ a subsequence s.t. $v_{i_{k}}-v_{i_{l}}\rightarrow0$
in $\mathscr{H}_{E}$, as $d\left(x_{i_{k}},x_{i_{l}}\right)\xrightarrow[\; k,l\rightarrow\infty]{}0$. 
\end{proof}

\section{Poisson-representations}

Let $G=\left(V,E\right)$ be as above, and let $c:E\rightarrow\mathbb{R}_{+}$
be a fixed conductance function. Let $d=d_{res}$ be the corresponding
resistance metric.

Our standard assumptions on $G,c$ are as outlined in Section \ref{sub:setting}
above.

We assume in addition that 
\begin{enumerate}
\item \label{enu:as1}$\#V=\aleph_{0}$, i.e., countable infinite.
\item \label{enu:as2}$d_{res}$ is bounded on $V\times V$.
\item \label{enu:as3}For all $x\in V$, $\exists\varepsilon=\varepsilon_{x}$
s.t. 
\begin{equation}
\left\{ y\in V\:|\: d\left(x,y\right)<\varepsilon_{x}\right\} =\left\{ x\right\} ,\;\text{the singleton}.\label{eq:p2}
\end{equation}
 
\end{enumerate}
We shall denote by $M$ the metric completion of $\left(V,d_{res}\right)$,
and identify $V$ as a subset of $M$ in the usual way, where 
\begin{equation}
x\in V\longleftrightarrow\underset{(\infty\:\text{repetition of vertex \ensuremath{x}})}{\mbox{class}\left(x,x,x,\cdots\right)\in M}\label{eq:p3}
\end{equation}

\begin{prop}
For $n\in\mathbb{N}$, set $w=\left(z_{1},\cdots,z_{n}\right)$ where
$z_{i}\in V$ (vertices), a finite word, and denote by $\left(w\underline{x}\right)$
the concatenation sequence 
\begin{align}
(z_{1},z_{2},\cdots,z_{n}, & \underset{\text{infinite repetition}}{\underbrace{x,x,x,x,x,x\cdots)};}\label{eq:w1}
\end{align}
we set $\underline{x}=\left(x,x,x,x,\cdots\right)$; then $\gamma\left(x\right)=\left\{ \underline{x}\right\} \cup\left\{ w\underline{x}\right\} $,
as $w$ ranges over all finite words.\end{prop}
\begin{proof}
If $\left(y_{i}\right)_{i=1}^{\infty}$ is a sequence of vertices
s.t. $\lim_{i\rightarrow\infty}d\left(y_{i},x\right)=0$, then, since
$x$ is isolated by (\ref{enu:as3}), see (\ref{eq:p2}), there must
be a $n\in\left\{ 0,1,2,\cdots\right\} $ such that $y_{i}=x$ for
all $i\geq n$; and the desired conclusion follows.\end{proof}
\begin{thm}
\label{thm:poisson}Let $G=\left(V,E\right)$, $c$, $d_{res}$ satisfying
the conditions above, including (\ref{enu:as1})-(\ref{enu:as3})
(so $d_{res}$ is bounded). Then 
\begin{equation}
B:=M\backslash V\label{eq:p4}
\end{equation}
is closed in $M$; and for every $x\in V$, there is a Borel probability
measure $\mu_{x}$ on $B$, i.e., $\mu_{x}\in M_{1}\left(B\right)$
such that, for all harmonic functions $h$ on $V$ with $\left\Vert h\right\Vert _{\mathscr{H}_{E}}<\infty$,
we have 
\begin{equation}
h\left(x\right)=\int_{B}\widetilde{h}\left(b\right)d\mu_{x}\left(b\right)\label{eq:p5}
\end{equation}
where $\widetilde{h}$ is the extension $\in C\left(M\right)$ of
$h$, obtained by metric completion, and where the function on the
RHS in (\ref{eq:p5}) is $\widetilde{h}\big|_{B}$.\end{thm}
\begin{proof}
By Corollary \ref{cor:fext}, every $f\in\mathscr{H}_{E}$ has a unique
continuous extension $\widetilde{f}$ to $M$; and 
\begin{equation}
\left|\widetilde{f}\left(b\right)-\widetilde{f}\left(b'\right)\right|^{2}\leq d\left(b,b'\right)\left\Vert f\right\Vert _{\mathscr{H}_{E}}^{2}\label{eq:p6}
\end{equation}
holds for $\forall b,b'\in M$. By (\ref{enu:as3}), (Section \ref{sub:mcom}),
$V$ identifies as an open subset in $M$, and so $B=M\backslash V$
is closed; and therefore compact. Recall $M$ is compact by Theorem
\ref{thm:Mcpt}.

Recall from Section \ref{sub:setting}, that a function $h$ on $V$
is harmonic iff $\mathbb{P}h=h$, where 
\begin{equation}
\left(\mathbb{P}h\right)\left(x\right)=\sum_{y\sim x}p_{xy}h\left(y\right)\label{eq:p8a}
\end{equation}
and $p_{xy}:=c_{xy}/c\left(x\right)$, for $\forall\left(xy\right)\in E$.
Also recall, $\left(\Delta f\right)\left(x\right)=\sum_{y\sim x}c_{xy}\left(f\left(x\right)-f\left(y\right)\right)$. 

Hence the harmonic functions $h$ in $\mathscr{H}_{E}\left(\subset C\left(M\right)\right)$
satisfy
\begin{equation}
\sup_{x\in V}\left|h\left(x\right)\right|=\left\Vert \widetilde{h}\big|_{B}\right\Vert _{\infty}.\label{eq:p8b}
\end{equation}
This is an application of (\ref{eq:p8a}) and a simple maxmin-principle. 

Now set $\mathcal{A}\subset C\left(B\right)$ as follows:
\begin{equation}
\mathcal{A}=\left\{ \widetilde{h}\big|_{B}\:;\:\mathbb{P}h=h,\: h\in\mathscr{H}_{E}\right\} \label{eq:9}
\end{equation}
where ``$\big|_{B}$'' denotes restriction; then, for every $x\in V$,
the point-evaluation mapping:
\begin{equation}
\mathcal{A}\ni\widetilde{h}\big|_{B}\longmapsto h\left(x\right)\label{eq:9b}
\end{equation}
defines a positive linear functional. Since $\mathbb{P}\left(\mathbbm{1}\right)=1$
where $\mathbbm{1}$ is the constant one function, it follows that
$\mathbbm{1}\in\mathcal{A}$, and that $\mathbbm{1}\mapsto1$ in (\ref{eq:9b})
(i.e., the functional in (\ref{eq:9b}) attains value 1 on the constant
function \textquotedblleft one.\textquotedblright )

By the extension theorem of Banach and Krein, there is a positive
linear functional on all of $C\left(B\right)$ which extends (\ref{eq:9b})
from $\mathcal{A}$. By Riesz' theorem, it is given by a unique probability
measure $\mu_{x}\in M_{1}\left(B\right)$. Restricting to $\mathcal{A}$,
and using (\ref{eq:p8b}), we get the desired formula (\ref{eq:p5});
i.e., $\mu_{x}$ is the Poisson-kernel, and $B$ is a Poisson-boundary,
i.e., it reproduces the harmonic functions in $\mathscr{H}_{E}$.
\end{proof}

\section{Continuous vs discrete: Examples}
\begin{rem}
\label{rem:sp}The orthogonal splitting 
\begin{equation}
\mathscr{H}_{E}=Harm\oplus\overline{span}^{\tiny\mathscr{H}_{E}\mbox{-norm}}\left\{ \delta_{x}\:\big|\: x\in V\right\} \label{eq:cd1}
\end{equation}
is often called the \emph{Royden-decomposition} (see e.g., \cite{MR674604,MR0225999}).
There is a continuous analogy: 

Let $\Omega\subset\mathbb{R}^{d}$ be bounded, and set 
\begin{equation}
\begin{split} & H^{1}\left(\Omega\right)=\left\{ f\:\mbox{on}\:\Omega\:\big|\:\nabla f\in L^{2}\left(\Omega\right)\right\} /\mbox{constants}\\
 & \left\Vert f\right\Vert _{H^{1}\left(\Omega\right)}^{2}=\int_{\Omega}\left|\nabla f\right|^{2}dx,
\end{split}
\label{eq:cd2}
\end{equation}
and set
\begin{equation}
Harm_{\Omega}=\left\{ f\in H^{1}\left(\Omega\right)\:\big|\:\Delta f=0\right\} .\label{eq:cd3}
\end{equation}
Then 
\[
H^{1}\left(\Omega\right)=Harm{}_{\Omega}\oplus\overline{C_{c}^{\infty}\left(\Omega\right)}^{\tiny H_{1}\left(\Omega\right)\mbox{-closure}};
\]
i.e., we have the implication
\[
\boxed{f\in H^{1}\left(\Omega\right),\:\mbox{and}\;\left\langle f,\varphi\right\rangle _{H^{1}\left(\Omega\right)}=0,\:\forall\varphi\in C_{c}^{\infty}\left(\Omega\right)}\Longrightarrow\Delta f=0.
\]

\end{rem}

\subsection{Continuous models}
\begin{lem}
Let 
\begin{align}
\mathscr{H}_{E} & =\left\{ f:\mathbb{R}\rightarrow\mathbb{C}\;|\;\text{measurable},\: f\in L^{2},f'\in L^{2}\right\} \label{eq:e1}\\
\left\Vert f\right\Vert _{\mathscr{H}_{E}}^{2} & =\frac{1}{2}\left(\int_{\mathbb{R}}\left|f\right|^{2}+\int_{\mathbb{R}}\left|f'\right|^{2}\right);\label{eq:e2}
\end{align}
where $f'$ in (\ref{eq:e1}) denotes the weak-derivative of $f$. 
\begin{enumerate}[label=(\roman{enumi}),ref=\roman{enumi}]
\item Then $\mathscr{H}_{E}$ is a reproducing kernel Hilbert space (RKHS)
consisting of bounded continuous functions. 
\item \label{enu:c2}Moreover, $\mathscr{H}_{E}$ is an algebra under pointwise
product with 
\[
\left\Vert fg\right\Vert _{\mathscr{H}_{E}}\leq\text{Const}\left\Vert f\right\Vert _{\mathscr{H}_{E}}\left\Vert g\right\Vert _{\mathscr{H}_{E}},\quad\forall f,g\in\mathscr{H}_{E}.
\]

\end{enumerate}
\end{lem}
\begin{proof}
$\mathscr{H}_{E}$ is a RKHS with kernel 
\begin{equation}
K\left(x,y\right)=e^{-\left|x-y\right|}.\label{eq:e3}
\end{equation}
To see this, set 
\begin{equation}
K_{x}\left(\cdot\right)=e^{-\left|x-\cdot\right|},\label{eq:e4}
\end{equation}
and one checks that 
\begin{align*}
\left\langle K_{x},f\right\rangle _{\mathscr{H}_{E}} & =\frac{1}{2}\left(\int_{\mathbb{R}}K_{x}f+\int_{\mathbb{R}}K_{x}'f'\right)=f\left(x\right),\quad\forall f\in\mathscr{H}_{E},\;\forall x\in\mathbb{R}.
\end{align*}

For a proof of part (\ref{enu:c2}), see \cite{Jor81}.

Note that $K_{x}''=K_{x}-2\delta_{x}$ in the sense of distribution.
See Fig \ref{fig:esp}.
\end{proof}
\begin{figure}[H]
\begin{tabular}{cc}
\includegraphics[width=0.3\columnwidth]{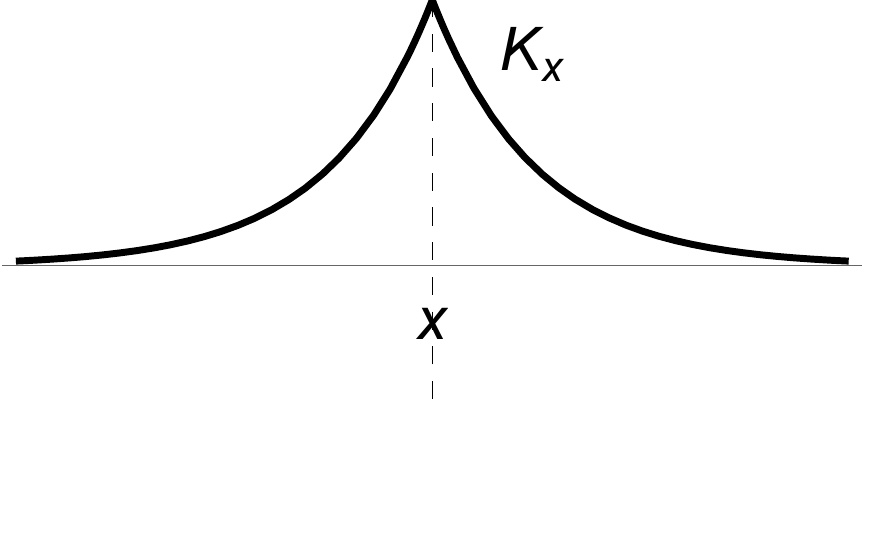} & \includegraphics[width=0.3\columnwidth]{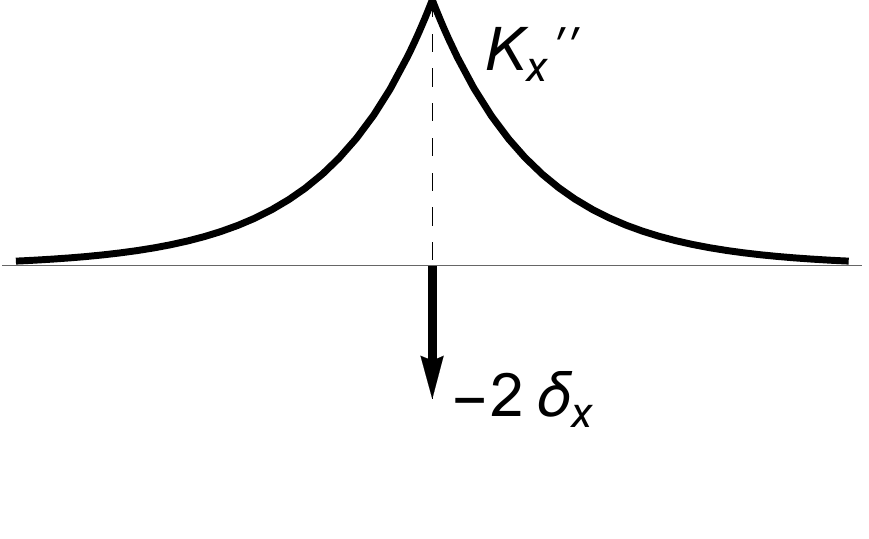}\tabularnewline
\includegraphics[width=0.3\columnwidth]{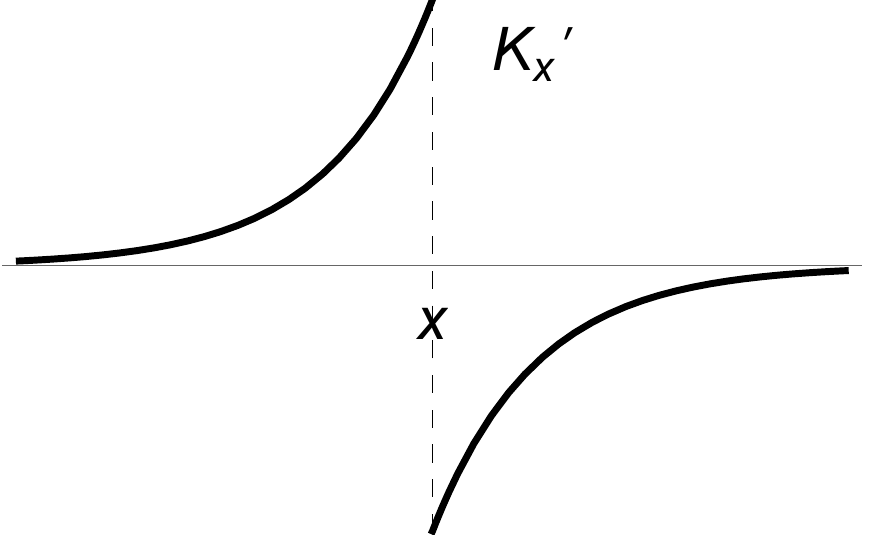} & \tabularnewline
\end{tabular}

\protect\caption{\label{fig:esp}$K_{x}''=K_{x}-2\delta_{x}$}
\end{figure}

\begin{lem}
If $f\in\mathscr{H}_{E}$, then $f$ is bounded, and $\lim_{\left|x\right|\rightarrow\infty}f\left(x\right)=0.$\end{lem}
\begin{proof}
This follows from Riemann-Lebesgue, since 
\[
f\left(x\right)=\frac{1}{2\pi}\int_{\mathbb{R}}e^{ix\xi}\hat{f}\left(\xi\right)d\xi,\;\text{and}\;\hat{f}\in L^{1}\left(\mathbb{R}\right)
\]
with $\sup_{x\in\mathbb{R}}\left|f\left(x\right)\right|\leq\left\Vert f\right\Vert _{\mathscr{H}_{E}}$,
for all $f\in\mathscr{H}_{E}$.
\end{proof}
The resistance distance in this case is
\begin{equation}
\begin{split} & d\left(x,y\right)=\left\Vert K_{x}-K_{y}\right\Vert _{\mathscr{H}_{E}}^{2}=2\left(1-e^{-\left|x-y\right|}\right),\;\mbox{and}\\
 & \sup_{x,y\in\mathbb{R}}d\left(x,y\right)\leq2.
\end{split}
\label{eq:e6}
\end{equation}
Hence the resistance metric $d$ in (\ref{eq:e6}) is bounded on $\mathbb{R}$,
and the completion of $\mathbb{R}$ with respect to $d$ is the one-point
compactification of $\mathbb{R}$, but for \emph{discrete models}:

\subsection{Discrete Models}

Let $G=\left(V,E,c\right)$ be a weighted graph, with vertex-set $V$,
edges $E$, and a fixed conductance function $c$. Let $d=d_{res}$
be the resistance metric, and we study the metric completion of $G$.

For functions on the $\mathbb{Z}$-lattice $L_{d}:=\mathbb{Z}^{d}$,
$d\geq1$, fixed, set 
\begin{equation}
\left\Vert f\right\Vert _{\mathscr{H}_{E}}^{2}=\frac{1}{2}\underset{x\sim y}{\sum\sum}e^{\left|x-y\right|}\left|f\left(x\right)-f\left(y\right)\right|^{2};\label{eq:c0}
\end{equation}
where $x\sim y$, $x\neq y$, in (\ref{eq:c0}) denotes \emph{nearest
neighbors}; and $x=\left(x_{1},x_{2},\cdots,x_{d}\right)\in\mathbb{Z}^{d}$.
(See Fig \ref{fig:lattice} for the case of $d=2$ and $d=3$.) Let
\begin{equation}
\mathscr{H}_{E}=\left\{ f\;\mbox{on}\;\mathbb{Z}^{d}\:|\:\left\Vert f\right\Vert _{\mathscr{H}_{E}}<\infty\right\} \label{eq:c1}
\end{equation}
we set 
\begin{equation}
\left|x-y\right|:=\left(\sum_{i=1}^{d}\left|x_{i}-y_{i}\right|^{2}\right)^{\frac{1}{2}},\quad\forall x,y\in\mathbb{Z}^{d}.\label{eq:c2}
\end{equation}

\begin{figure}
\begin{tabular}{cc}
\includegraphics[width=0.3\textwidth]{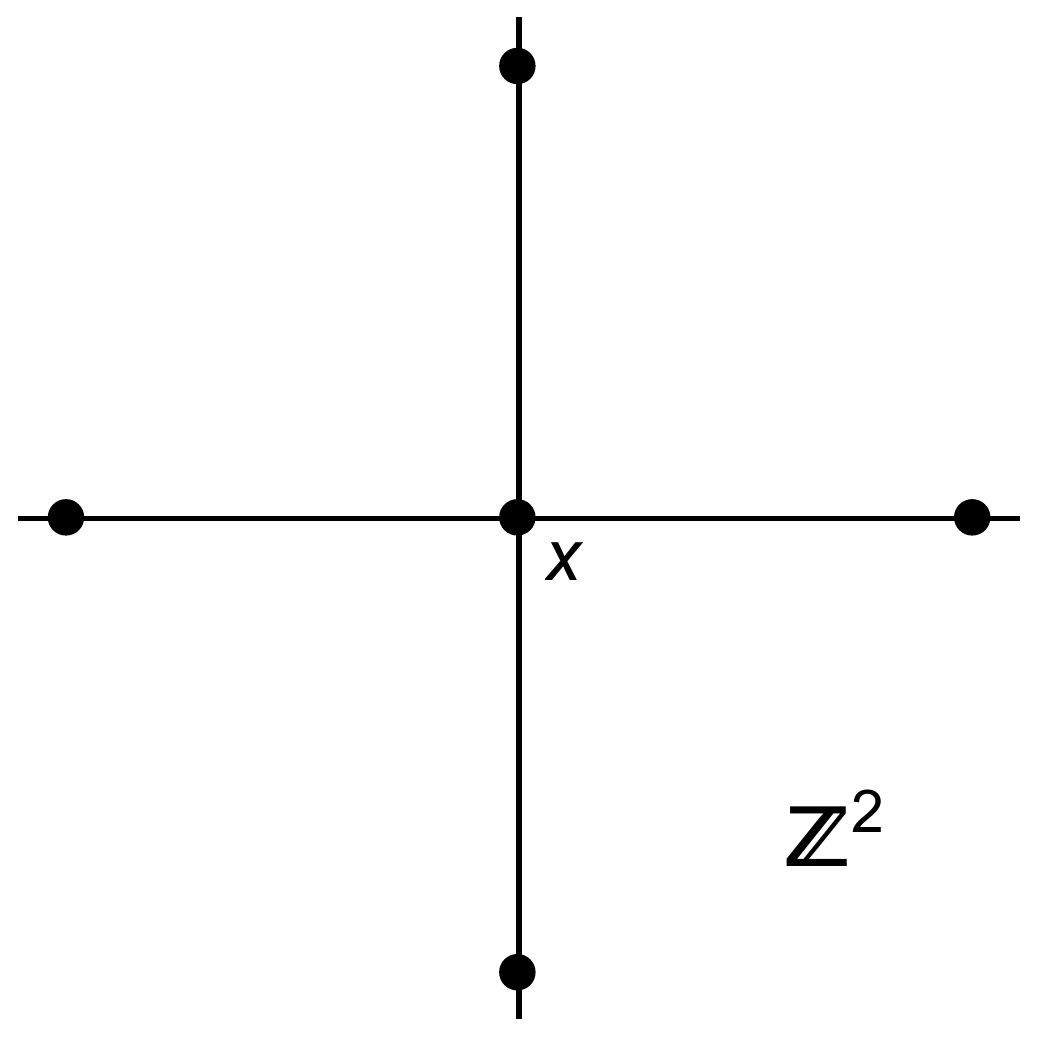} & \includegraphics[width=0.3\textwidth]{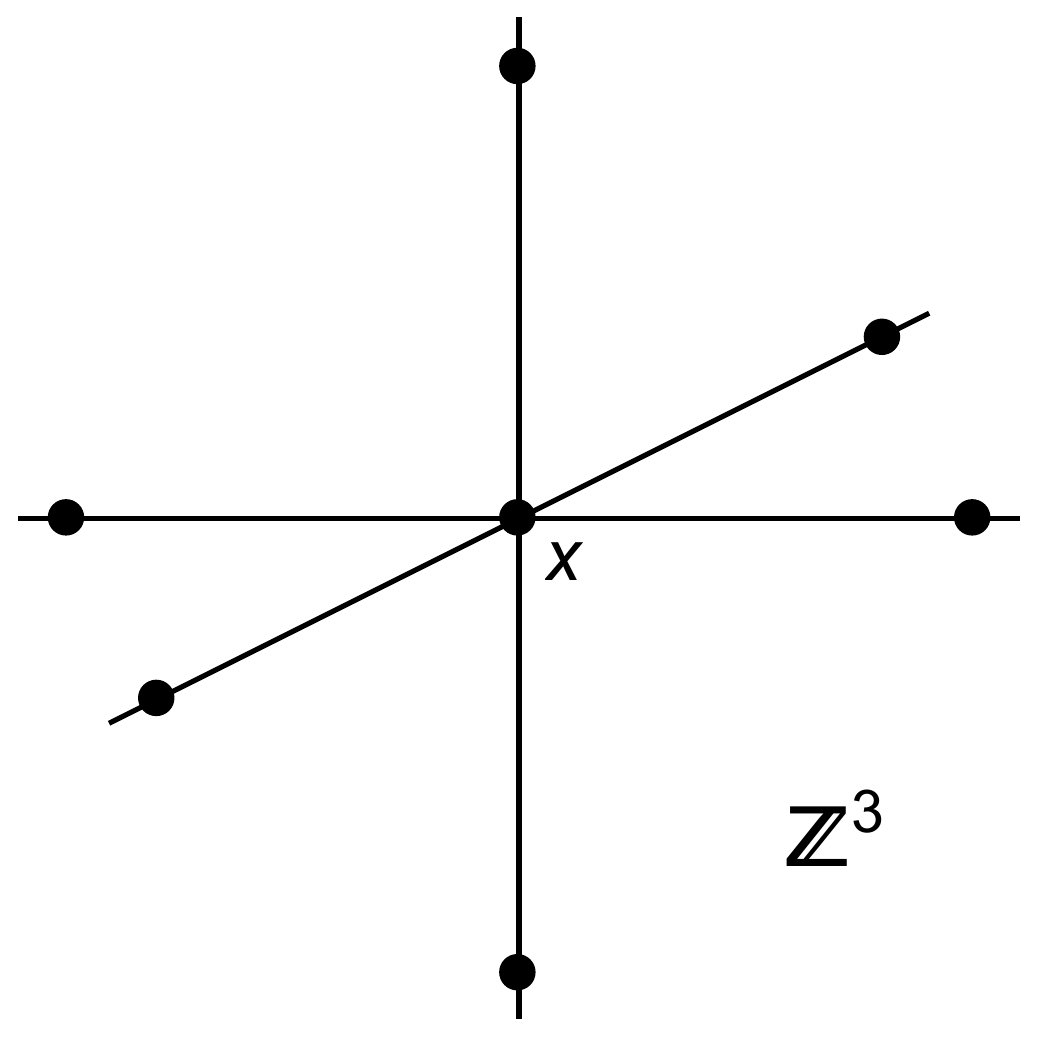}\tabularnewline
\end{tabular}

\protect\caption{\label{fig:lattice}Nearest neighbors for the lattices $\mathbb{Z}^{d}$,
for $d=2$ and $d=3$.}

\end{figure}

\begin{lem}
For $\forall x\in\mathbb{Z}^{d}$, we have the following: $\exists K=K_{x}<\infty$
s.t. 

\begin{equation}
\left|f\left(x\right)-f\left(y\right)\right|^{2}\leq K_{x}\left\Vert f\right\Vert _{\mathscr{H}_{E}}^{2}\;(\mbox{\textup{see Theorem \ref{thm:metric}}}.)\label{eq:c3}
\end{equation}
\end{lem}
\begin{proof}
Let $x\sim y$ denote nearest neighbors. Then $\exists!$ $i$ s.t.
$\left|x_{i}-y_{i}\right|=1$, so $x_{j}-y_{j}=0$ for $j\neq i$.
The proof of (\ref{eq:c3}) is standard.
\end{proof}
Set 
\[
\left(\Delta f\right)\left(x\right)=\sum_{y\sim x}e^{\left|x-y\right|}\left(f\left(x\right)-f\left(y\right)\right),\quad\forall f\in\mathscr{H}_{E}.
\]

\begin{example}
For $d=1$, consider $\mathbb{Z}_{+}$ (see Fig \ref{fig:z1c}), and
\[
p_{+}=\frac{e}{1+e},\quad p_{-}=\frac{1}{1+e}.
\]

A function $u$ on $\mathbb{Z}_{+}$ is harmonic iff $I_{x}:=e^{x}\left(u_{x+1}-u_{x}\right)$
is constant; and 
\begin{align*}
\left\Vert u\right\Vert _{\mathscr{H}_{E}}^{2} & =\sum_{x}e^{x}\left(u_{x+1}-u_{x}\right)^{2}=I_{1}^{2}\sum_{x}e^{-x}=\frac{I_{1}^{2}}{e-1}<\infty.
\end{align*}
Fix $0<x<y$, then $v_{xy}=v_{yo}\left(t\right)-v_{xo}\left(t\right)$,
where 
\begin{align*}
v_{yo}\left(t\right) & =\begin{cases}
\sum_{i\leq y}e^{-i} & \mbox{if}\; t\leq y\\
\sum_{i=1}^{y}e^{-i} & \mbox{if}\; t>y
\end{cases}\\
v_{xy}\left(t\right) & =\begin{cases}
0 & \mbox{if}\;0<t\leq x\\
\sum_{i=x+1}^{t}e^{-i} & \mbox{if}\; x<t\leq y\\
\sum_{i=x+1}^{y}e^{-i} & \mbox{if}\; y\leq t,\quad t\in\mathbb{Z}_{+}
\end{cases}
\end{align*}
and 
\[
d_{res}\left(x,y\right)=\sum_{i=x+1}^{y}e^{-i}=\frac{e^{-x}-e^{-y}}{e-1};
\]
and so $d_{res}$ is clearly bounded. 

But in this case the metric compactification is just the one-point
compactification: 
\[
d_{res}\left(x,\infty\right)=\frac{e^{-x}}{e-1};\quad x\in\mathbb{Z}_{+}.
\]

\end{example}
It follows, in these examples, that $B=M\backslash V$ is a singleton;
so $M$ is the one-point compactification. 

\begin{figure}
\includegraphics[width=0.4\textwidth]{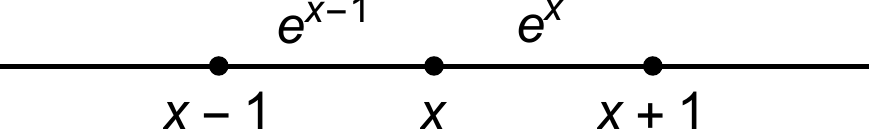}

\protect\caption{\label{fig:z1c}$\mathbb{Z}_{+}$ conductance function $c$, $c_{x,x+1}=e^{x}$,
$x\in\mathbb{Z}_{+}$.}
\end{figure}

\begin{example}
\label{exa:Zlat}$V=\mathbb{Z}_{+}^{d}$, $d>1$. Set $x+\varepsilon_{i}:=\left(x_{1},\cdots,x_{i}+1,x_{i+1}\cdots,x_{d}\right)$,
for all $x\in\mathbb{Z}^{d}$. In this case we have: 
\begin{align*}
\left(\Delta f\right)\left(x\right) & =\sum_{i=1}^{d}e^{\left|x+\varepsilon_{i}\right|}\left(f\left(x\right)-f\left(x+\varepsilon_{i}\right)\right);\;\mbox{and}\\
\left\Vert f\right\Vert _{\mathscr{H}_{E}}^{2} & =\sum_{x,i}e^{\left|x+\varepsilon_{i}\right|}\left|f\left(x\right)-f\left(x+\varepsilon_{i}\right)\right|^{2}\\
 & \sim\sum_{x\in\mathbb{Z}_{+}^{d}}e^{\left|x\right|}\sum_{i=1}^{d}\left|f\left(x+\varepsilon_{i}\right)-f\left(x\right)\right|^{2}\longrightarrow0
\end{align*}
and $\left\Vert f\right\Vert _{\mathscr{H}_{E}}^{2}<\infty\Longrightarrow$
\[
\exists N\quad\mbox{s.t.}\quad e^{\left|x\right|}\sum_{i}\left|f\left(x+\varepsilon_{i}\right)-f\left(x\right)\right|^{2}<1,\quad\forall x<\mathbb{Z}_{+}^{d},\:\forall\left|x\right|>N.
\]
See Fig \ref{fig:nbh3}.

Hence, 
\[
\sum_{i}\left|f\left(x+\varepsilon_{i}\right)-f\left(x\right)\right|^{2}<e^{-\left|x\right|},\;\text{asymptotic as \ensuremath{\left|x\right|\rightarrow\infty}};
\]
 so that 
\[
f\left(x\right)\sim\mbox{Const}\cdot\left(1-e^{-\left|x\right|}\right),\;\text{asymptotic as}\;\left|x\right|\rightarrow\infty\;\mbox{in}\;\mathbb{Z}_{+}^{d}.
\]

Let $x\in L_{d}=\mathbb{Z}_{+}^{d}$, and 
\[
V_{x}\sim e^{-\left|x-t\right|},\quad\forall t\in L_{d}
\]
so that 
\[
\left\langle v_{x},f\right\rangle =f\left(x\right)-f\left(o\right),\quad\forall x\in L_{d},\;\forall f\in\mathscr{H}_{E}.
\]

If $d\geq3$, one checks that 
\[
\left\langle v_{x},v_{y}\right\rangle _{\mathscr{H}_{E}}\simeq\frac{1}{\left|x-y\right|^{d-2}};
\]
and as a result, is a one-point compactification, i.e., $B=M\backslash L_{d}=\left\{ \infty\right\} $
the point at ``infinity.''
\end{example}
\begin{figure}[H]
\includegraphics[width=0.4\textwidth]{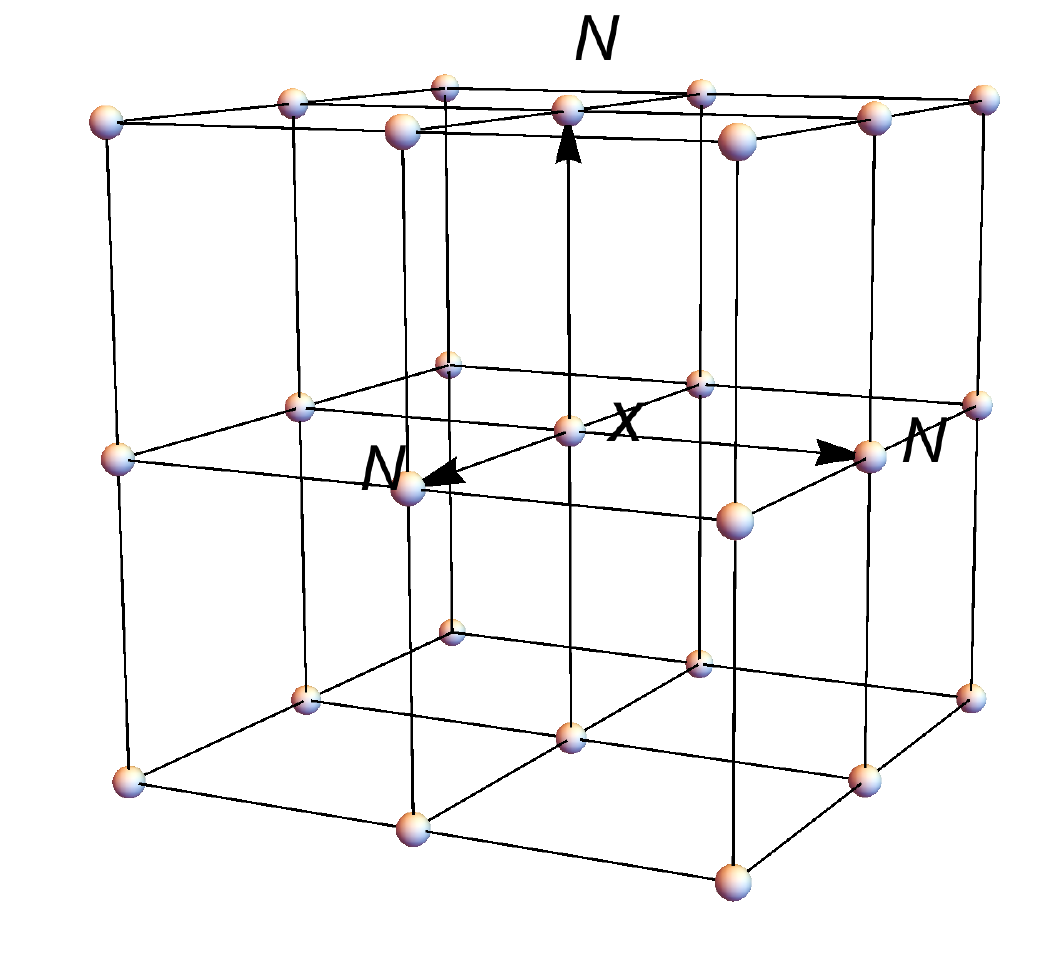}

\protect\caption{\label{fig:nbh3}$d=3$}
\end{figure}

\begin{example}
Let $V=$ the binary tree, see Fig\ref{fig:bt}. If a vertex $x$
in the tree is at level $n$, set 
\[
c_{\left(x,x+\right)}=c_{+}\left(n\right),\quad c_{\left(x,x-\right)}=c_{-}\left(n\right).
\]
Then the arguments from above show that if $\sum_{n=1}^{\infty}\frac{1}{c_{\pm}\left(n\right)}<\infty$,
then $B:=M\backslash V$ is a Cantor-space.
\end{example}
\begin{figure}
\includegraphics[width=0.5\textwidth]{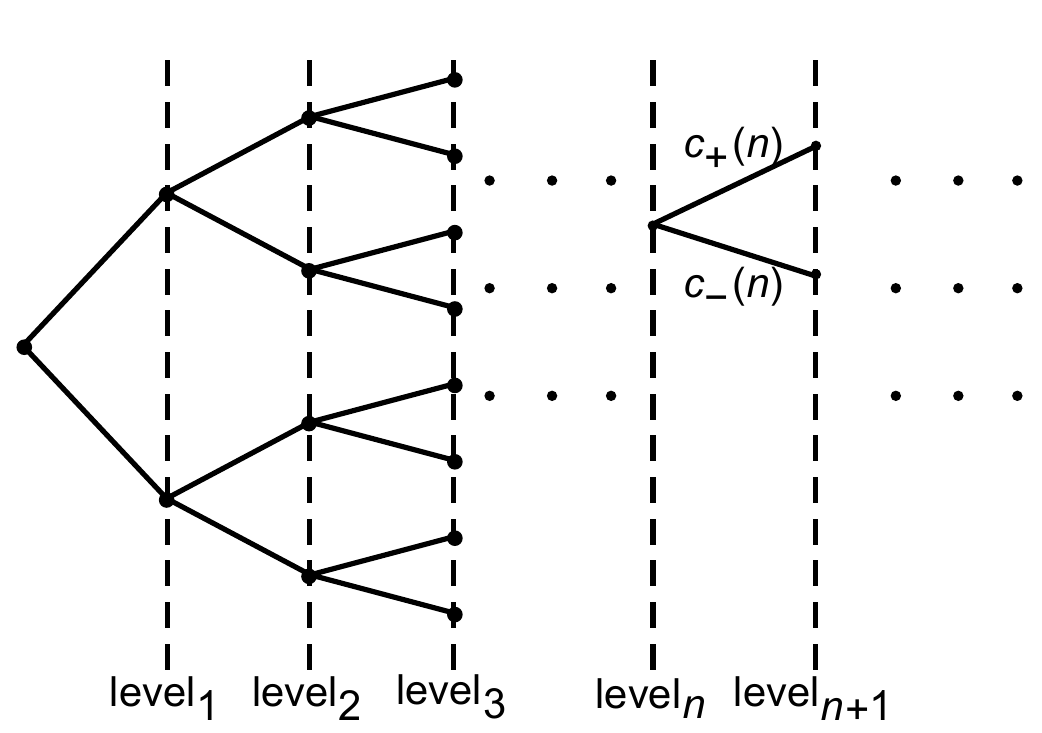}

\protect\caption{\label{fig:bt}Binary tree with conductance}

\end{figure}

\subsection{Bratteli diagrams}

In our present papers, we considered weighted graphs $G=(V,E,c)$,
vertices, edges and a weight (conductance) function. A Bratteli diagram
is a special case of this, but the weighting usually doesn't refer
to a conductance, but rather some kind of counting. In detail, if
$G$ is a Bratteli diagram, then its vertex set is stratified, by
finite subsets $V_{n}$, called levels. While $V$ is infinite, the
sets $V_{n}$ are finite. Then the requirement on $G$ to be a Bratteli
diagram is that the edges (lines in $E$) connect vertices from $V_{n}$
to those at different levels; the nearest neighbor vertices are from
level $n-1$, and level $n+1$. In its initial form (see \cite{MR0312282})
the Bratteli diagrams (later terminology) served as classification
labels for approximately finite-dimensional $C^{*}$-algebras (also
called AF-algebras, more precisely inductive limits of matrix-algebras).
The need for such classification was initially motivated by physics.
Subsequently, and initiated by George Elliott), the Bratteli diagrams
acquired the structure of ordered groups (called $K$-groups), and
the classification problem eventually took a rather complete form.
But in the spirit of the original use of the diagrams from \cite{MR0312282},
they have found many other uses in representation theory; the fundamental
idea being that the lines (edges) are effective in classifying complicated
systems of inclusions, i.e., counting the respective multiplicities
in these inclusions of algebras, or representations, by numbers assigned
to the edges. In this incarnation, they are even known as useful tools
in the design of fast (finite) Fourier transforms.

And there are yet other applications; some deal with symbolic dynamics;
see the papers in the bibliography, for example \cite{MR1194074},
and measures on infinite path spaces obtained from ``infinite strings
of edges'' from the given Bratteli diagram.

The papers \cite{MR1804950} and \cite{MR2030387} deal with yet a
different classification; that of order-isomorphism of the diagrams
themselves. It turns out that the latter classification problem, in
the general case, is so ``complicated'' that it has been proved
to be undecidable. So in the Bratteli-Jorgensen et al. papers regarding
this, we narrowed our focus to that of stationary Bratteli-diagrams;
and we proved that then a classification is possible; even by explicit
algorithms, and by explicit lists of numerical parameters.

Nonetheless the questions we consider here fall in a different category,
and they don't restrict the focus to stationary diagrams; even apply
to graphs $G$ which are not Bratteli diagrams.

If $\Delta=C-E$ as an $\infty\times\infty$ matrix representation,
where 
\begin{align*}
C & =\mbox{diag}\left(c\left(x\right)\right)_{x\in V}=\begin{pmatrix}c\left(x\right) &  & \substack{\scalebox{2.5}{0}}
\\
 & \rotatebox{-45}{\text{\ensuremath{\cdots}\small diagonal\ensuremath{\cdots}}}\\
\substack{\scalebox{2.5}{0}}
\end{pmatrix}
\end{align*}
and
\begin{align*}
E & =\begin{pmatrix}0 & \ddots &  & \substack{\scalebox{2.5}{0}}
\\
\ddots & 0 & c_{xy}\\
 & c_{xy} & \ddots & \ddots\\
\substack{\scalebox{2.5}{0}}
 &  & \ddots
\end{pmatrix}
\end{align*}
symmetric, $c_{xy}>0$; then 
\begin{equation}
\Delta=\left(\Delta_{xy}\right)=C-E\label{eq:bd1}
\end{equation}
where 
\[
\Delta_{xy}=\begin{cases}
c\left(x\right) & \text{if \ensuremath{x=y}}\\
-c_{xy} & \text{if \ensuremath{\left(xy\right)\in E}}\;(\mbox{recall}\;\ensuremath{x\sim y\Longrightarrow x\neq y})\\
0 & \text{otherwise;}
\end{cases}
\]
and we get the Green's function $K$ as follows: 
\begin{equation}
K=\left\langle v_{x},v_{y}\right\rangle _{\mathscr{H}_{E}}\label{eq:bd2}
\end{equation}
the Green's function of $\Delta$ satisfies
\begin{equation}
\sum_{z}\Delta_{xz}K_{zy}=\delta_{xy},\label{eq:bd3}
\end{equation}
and 
\begin{eqnarray}
\Delta^{-1} & = & \left(C-E\right)^{-1}=\left(I-C^{-1}E\right)^{-1}C^{-1}\nonumber \\
 & = & \sum_{n=0}^{\infty}\left(C^{-1}E\right)^{n}C^{-1}=G_{P}C^{-1},\label{eq:bd4}
\end{eqnarray}
where $G_{P}$ is the Green's function of a Markov transition (see
Fig \ref{fig:mark}). Note that $C^{-1}$ is easy, since it is diagonal:
\begin{equation}
C=\mbox{diag}\left(\left(c\left(x\right)\right)_{x\in V}\right),\;\mbox{then}\quad C^{-1}=\mbox{diag}\left(\left(\frac{1}{c\left(x\right)}\right)_{x\in V}\right).\label{eq:bd5}
\end{equation}
An example is (see Fig \ref{fig:mark}-\ref{fig:Bd}) 
\begin{equation}
c\left(n\right)=c_{n}+c_{n+1},\quad c_{n}>0.\label{eq:bd6}
\end{equation}

\begin{figure}
\includegraphics[width=0.8\textwidth]{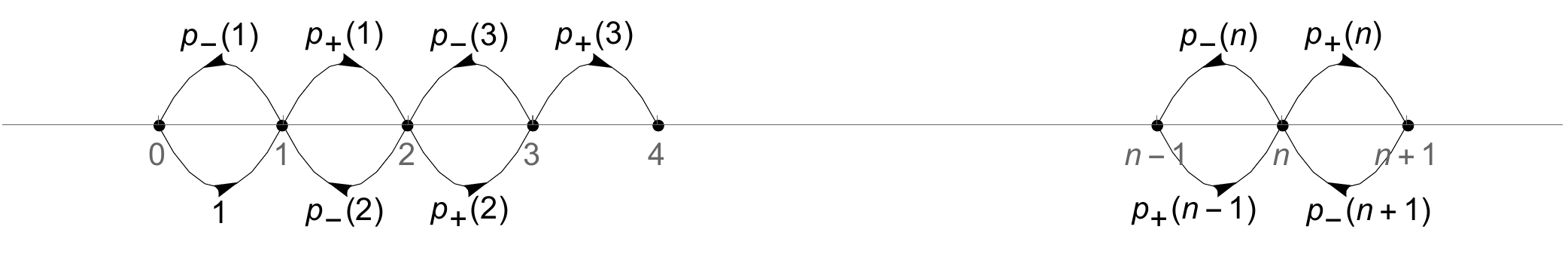}

\protect\caption{\label{fig:mark} Transition probabilities $p_{\pm}\left(n\right)$,
$n=0,1,2\cdots$}
\end{figure}

\begin{figure}
\includegraphics[width=0.7\textwidth]{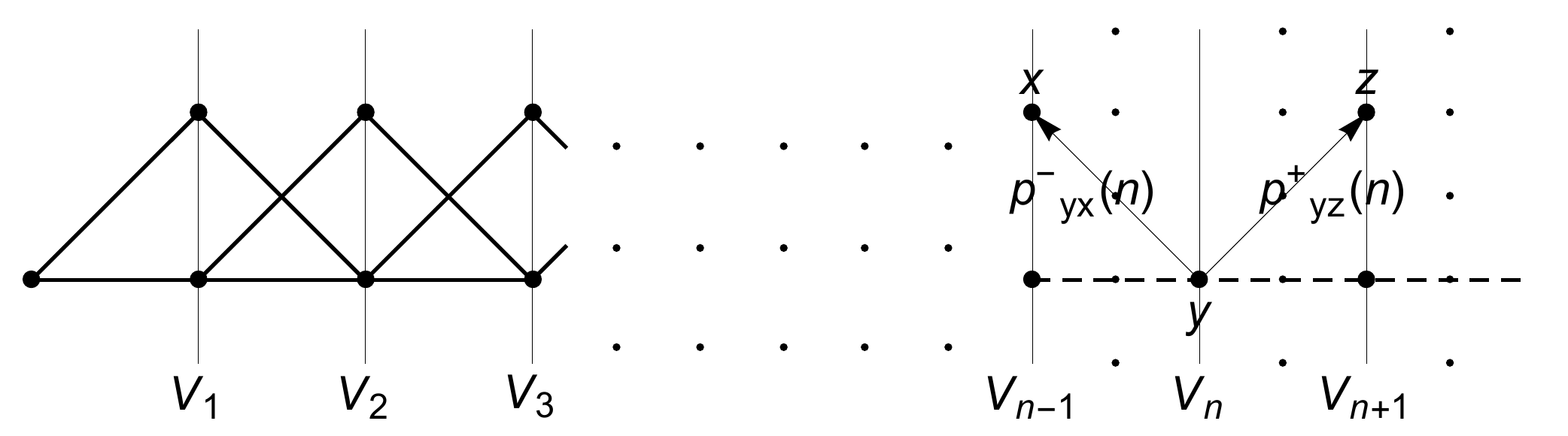}

\protect\caption{\label{fig:Bd}A Bratteli diagram, formula (\ref{eq:bd6}) with vertex-set
$V=\left\{ \emptyset\right\} \cup V_{1}\cup V_{2}\cup\cdots$ and
transition between neighboring levels.}
\end{figure}

\begin{lem}
\label{lem:K}If $\left(V,E,c\right)$ is constructed from a Bratteli
diagram with levels $V_{1},V_{2},\cdots$, then the Green's function
$K$ for $\Delta$ satisfies 
\begin{equation}
K=G_{P}C^{-1}\label{eq:bd10}
\end{equation}
where $G_{P}$ is the random-walk Green's function associated with
a $\pm$ Markov random walk, see (\ref{eq:bd5}) and Fig \ref{fig:mark}. 
\end{lem}
For Bratteli diagrams, see e.g., \cite{MR2030387,MR1804950,MR0312282,MR1710743,MR1194074};
and random walks, see e.g., \cite{MR3158629}. 
\begin{proof}[Proof of Lemma \ref{lem:K} (sketch)]
 Let $\left(p_{-}\left(n\right)\right)$ and $\left(p_{+}\left(n\right)\right)$
be the transition matrices
\begin{enumerate}
\item[] $\left(p_{-}\left(n\right)\right)_{xy}$: $x\in V_{n}$, $y\in V_{n-1}$,
transition from vertex on $V_{n}$ to $V_{n-1}$
\item[] $\left(p_{+}\left(n\right)\right)_{yz}$: $y\in V_{n}$, $z\in V_{n+1}$,
transition from vertex on $V_{n}$ to $V_{n+1}$, see Fig \ref{fig:trp},
\end{enumerate}

\begin{flushleft}
with row/column index picked from vertices in the respective levels. 
\par\end{flushleft}

The product of $C^{-1}E$ in (\ref{eq:bd4}) is then (see Fig \ref{fig:Bd1})
\begin{equation}
\left(C^{-1}E\right)_{xy}^{m}=\mbox{Prob}\left(\mbox{transition from vertex }x\:\mbox{to}\:\mbox{vertex}\: y\:\mbox{in time}\: m\right).\label{eq:bd13}
\end{equation}

\end{proof}
\begin{rem}
\label{rem:rwalk}Under the assumption in Theorem \ref{thm:cpt1}
and Theorem \ref{thm:poisson} one may show that in fact $B$ (see
(\ref{eq:p4})) is \emph{Martin-boundary} (see \cite{MR1463727,MR2863855})
for the random walk on $V$ defined by 
\begin{equation}
p_{xy}:=\frac{c_{xy}}{c\left(x\right)}\quad\mbox{for}\;\left(x,y\right)\in E.\label{eq:mb1}
\end{equation}
\end{rem}
\begin{proof}
(sketch) Let $G_{P}$ be the random-walk Green's function from (\ref{eq:bd4})
and Lemma \ref{lem:K}. Set 
\begin{equation}
K_{Martin}\left(x,y\right):=\frac{G_{P}\left(x,y\right)}{G_{P}\left(o,y\right)}.\label{eq:mb2}
\end{equation}
Then the argument from Theorem \ref{thm:cpt1} shows that $K_{Martin}\left(x,\cdot\right)$
extends to $B$, and that 
\begin{equation}
h\left(x\right)=\int_{B}\widetilde{h}\left(b\right)K_{Martin}\left(x,b\right)d\mu^{\left(Markov\right)}\left(b\right)\label{eq:mb3}
\end{equation}
holds for all $h\in Harm=\mathscr{H}_{E}\cap\left\{ h\::\:\Delta h=0\right\} =\mathscr{H}_{E}\cap\left\{ h\::\:\mathbb{P}h=h\right\} $. 
\end{proof}
\begin{figure}
\includegraphics[width=0.8\textwidth]{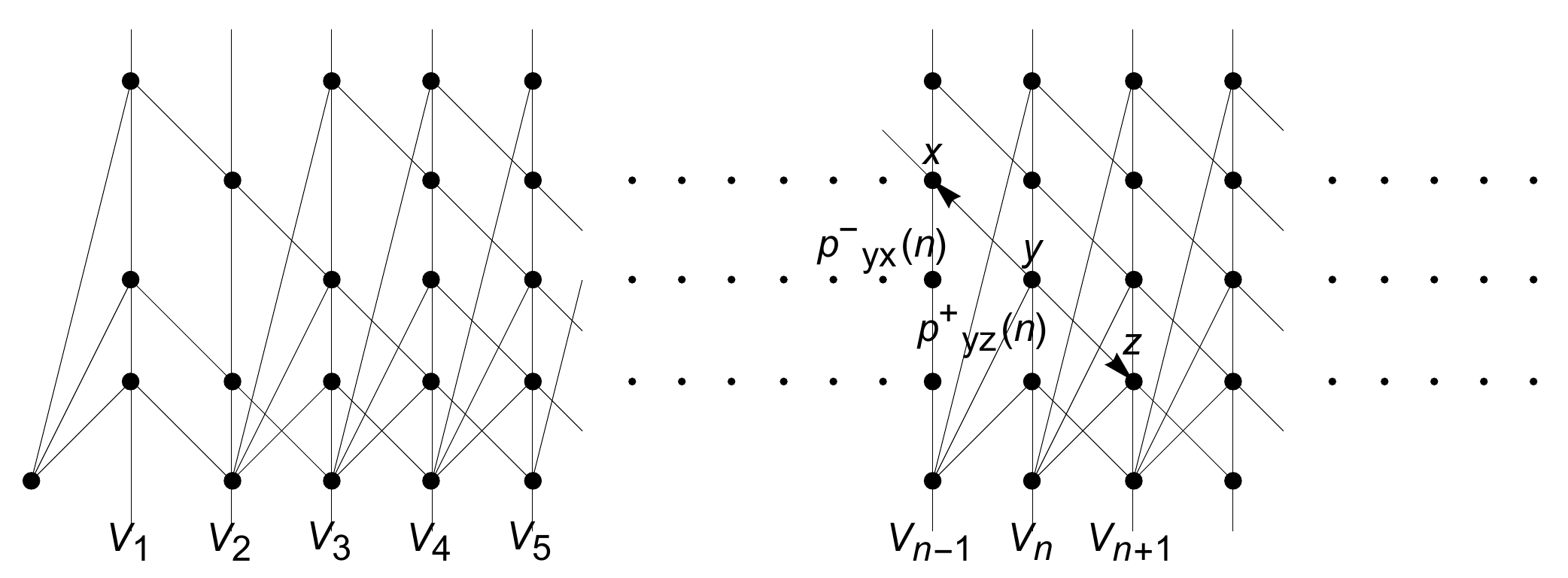}

\protect\caption{\label{fig:trp}}
\end{figure}

\begin{figure}
\includegraphics[width=0.4\textwidth]{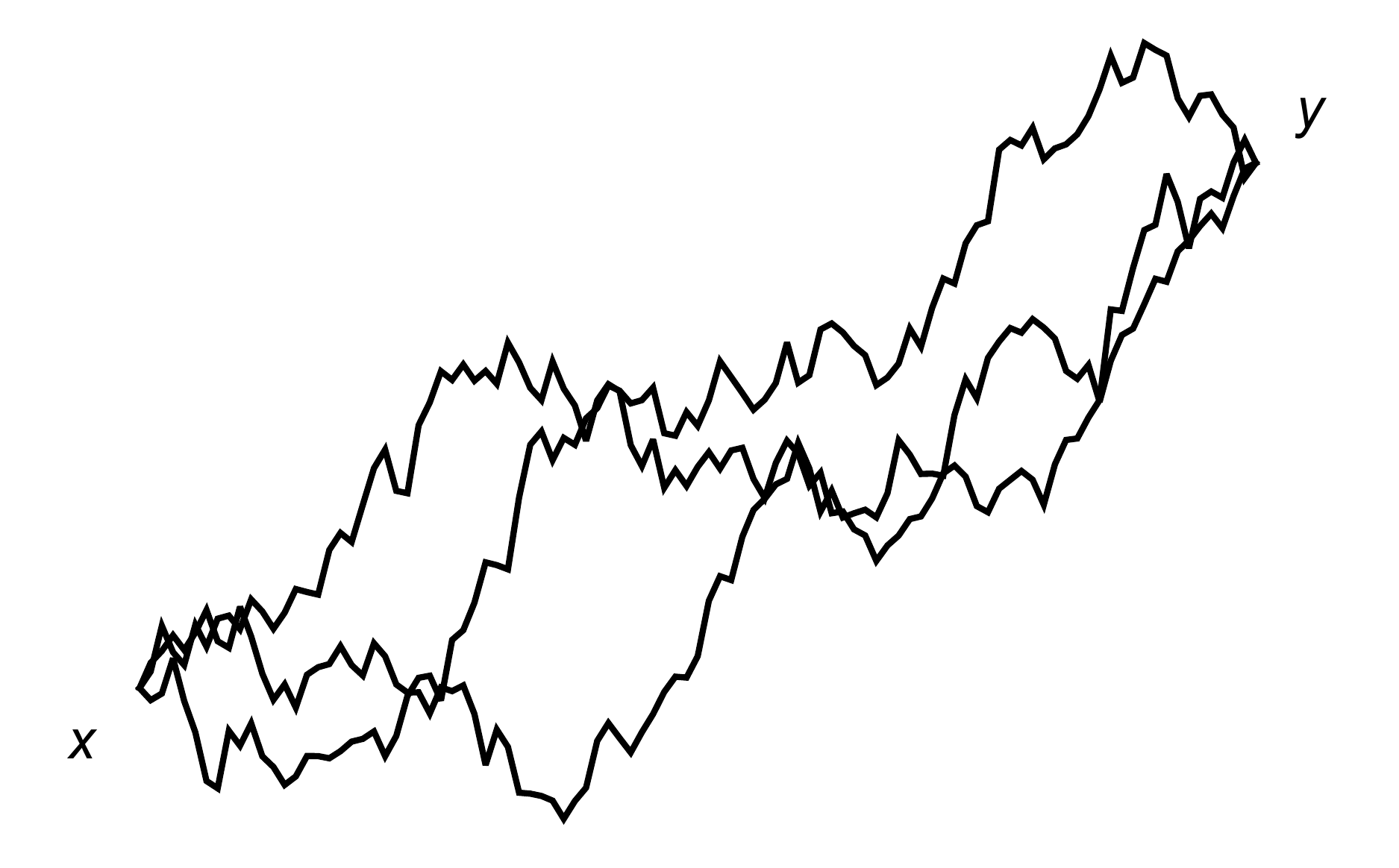}

\protect\caption{\label{fig:Bd1}Current flows from vertex $x$ to vertex $y$.}
\end{figure}

\begin{example}
\label{exa:btd}For the transition matrix $C^{-1}E=P$, computed with
the system in Fig \ref{fig:bt} of transition probabilities, we get
the following:
\begin{equation}
\begin{split}p_{i,i} & =0\\
p_{i,i+1} & =p_{+}\left(i\right),\;\mbox{and}\\
p_{i,i-1} & =p_{-}\left(i\right),\;\forall i\in\mathbb{Z},
\end{split}
\label{eq:tt1}
\end{equation}
with the remaining matrix-entries zero. For the computation of the
matrix powers $P^{m}$, $m=1,2,\cdots$, we make the following simplification:
$p_{+}\left(i\right)=p_{+}$, and $p_{-}\left(i\right)=p_{-}$. 

This then reduces to the following binomial model:
\begin{equation}
P=\begin{pmatrix}\rotatebox{-38}{\tiny\text{diagonal}}\\
\ddots & \ddots & \ddots\\
 & p_{-} & 0 & p_{+} &  & \substack{\scalebox{3}{0}}
\\
 &  & p_{-} & 0 & p_{+}\\
 &  &  & p_{-} & 0 & p_{+}\\
 & \substack{\scalebox{3}{0}}
 &  &  & \ddots & \ddots & \ddots\\
 &  &  &  &  &  & \rotatebox{-38}{\tiny\text{diagonal}}
\end{pmatrix},\label{eq:tt2}
\end{equation}
\begin{equation}
P^{2}=\begin{pmatrix}\ddots & \rotatebox{-35}{\tiny\text{diagonal}} & \ddots & \ddots &  &  & \substack{\scalebox{3}{0}}
\\
\ddots & 0 & 2p_{+}p_{-} & 0 & p_{+}^{2}\\
 & p_{-}^{2} & 0 & 2p_{+}p_{-} & 0 & p_{+}^{2}\\
 &  & p_{-}^{2} & 0 & 2p_{+}p_{-} & 0 & \ddots\\
\substack{\scalebox{3}{0}}
 &  &  & \ddots & \ddots & \rotatebox{-35}{\tiny\text{diagonal}} & \ddots
\end{pmatrix},
\end{equation}
and
\begin{equation}
P^{3}=\begin{pmatrix}\ddots & \ddots & \rotatebox{-35}{\tiny\text{diagonal}} & \ddots & \ddots & \ddots &  &  & \substack{\scalebox{3}{0}}
\\
\ddots & 0 & 3p_{+}p_{-}^{2} & 0 & 3p_{+}^{2}p_{-} & 0 & p_{+}^{3}\\
 & p_{-}^{3} & 0 & 3p_{+}p_{-}^{2} & 0 & 3p_{+}^{2}p_{-} & 0 & p_{+}^{3}\\
 &  & p_{-}^{3} & 0 & 3p_{+}p_{-}^{2} & 0 & 3p_{+}^{2}p_{-} & 0 & \ddots\\
\substack{\scalebox{3}{0}}
 &  &  & \ddots & \ddots & \ddots & \rotatebox{-35}{\tiny\text{diagonal}} & \ddots & \ddots
\end{pmatrix},
\end{equation}

\end{example}
Below we include a sample of matrix-entries for this binomial model:
\begin{itemize}
\item \emph{Even powers of the transition-matrix $P$}
\begin{eqnarray*}
P_{i,i+2k}^{2m} & = & \begin{pmatrix}2m\\
m-k
\end{pmatrix}p_{+}^{m+k}p_{-}^{m-k},\;\mbox{and}\\
P_{i,i-2k}^{2m} & = & \begin{pmatrix}2m\\
m-k
\end{pmatrix}p_{+}^{m-k}p_{-}^{m+k};
\end{eqnarray*}
where $k=0,1,\cdots,m$.
\item \emph{Odd powers of the transition-matrix $P$
\begin{eqnarray*}
P_{i,i+1+2k}^{2m+1} & = & \begin{pmatrix}2m+1\\
m-k
\end{pmatrix}p_{+}^{m+k+1}p_{-}^{m-k},\;\mbox{and}\\
P_{i,i-1-2k}^{2m+1} & = & \begin{pmatrix}2m+1\\
m-k
\end{pmatrix}p_{+}^{m-k}p_{-}^{m+k+1}.
\end{eqnarray*}
}
\end{itemize}
So for the $\infty\times\infty$ matrix $G_{P}$ in (\ref{eq:bd4})
we get:
\begin{eqnarray*}
\left(G_{P}\right)_{i,i+2k} & = & \sum_{m=0}^{\infty}\begin{pmatrix}2m\\
m-k
\end{pmatrix}p_{+}^{m+k}p_{-}^{m-k};\;\mbox{and}\\
\left(G_{P}\right)_{i,i+2k+1} & = & \sum_{m=0}^{\infty}\begin{pmatrix}2m+1\\
m-k
\end{pmatrix}p_{+}^{m+k+1}p_{-}^{m-k}.
\end{eqnarray*}
As a result, (\ref{eq:bd4}) yields an explicit formula for $K_{i,j}=\left\langle v_{i},v_{j}\right\rangle _{\mathscr{H}_{E}}$;
see (\ref{eq:bd4}) and (\ref{eq:bd2}).
\begin{thm}
The $\Delta$-Green's function $K$ in (\ref{eq:ta1}) has an explicit
(and closed form) expression; for example, its diagonal entries are:
\[
K_{i,i}=\frac{1}{c\left(i\right)\sqrt{1-4p_{+}\left(1-p_{+}\right)}}\quad\mbox{when}\; p_{+}\neq\frac{1}{2}.
\]
\end{thm}
\begin{proof}
The infinite sums used in computation of $\left(G_{P}\right)_{i,j}$;
and therefore of 
\begin{equation}
K_{i,j}=\left(G_{P}\right)_{i,j}/c\left(j\right)\label{eq:ta1}
\end{equation}
can be computed with the use of \emph{generating functions} for the
associated binomial coefficients. For example,
\begin{equation}
\sum_{n=0}^{\infty}\lambda^{m}\begin{pmatrix}2m\\
m
\end{pmatrix}=\frac{1}{\sqrt{1-4\lambda}},\;\mbox{setting}\;\lambda:=p_{+}p_{-};\label{eq:ta2}
\end{equation}
and so we get
\begin{equation}
\left(G_{P}\right)_{i,i}=\frac{1}{\sqrt{1-4p_{+}p_{-}}};\label{eq:ta3}
\end{equation}
and therefore
\begin{equation}
K_{i,i}=\frac{1}{c\left(i\right)\sqrt{1-4p_{+}\left(1-p_{+}\right)}}=\left\langle v_{i},v_{i}\right\rangle _{\mathscr{H}_{E}}=d_{res}\left(o,i\right),\label{eq:ta4}
\end{equation}
which is the desired conclusion.
\end{proof}
Note that to get absolute convergence in these series the requirement
on $p_{+}$ is that $p_{+}\in\left(0,\frac{1}{2}\right)\cup\left(\frac{1}{2},1\right)$.
(In this case, the resistance metric is bounded. We have $\sum_{j}\frac{1}{c\left(j\right)}<\infty$.)
The degenerate case is $p_{+}=p_{-}=\frac{1}{2}$. However the latter
degenerate case can easily be computed by hand. It is the case of
constant conductance function, $c_{i,i+1}=1$. 

For more details on this and related binomial models, see \cite{AJ15,AJSV13,MR3224442}.
\begin{rem}[On general Bratteli diagrams]
 While the formulas (\ref{eq:bd13})-(\ref{eq:ta4}) are derived
subject to rather restricting assumptions, an inspection of the arguments
shows that the ideas work for general Bratteli-diagrams; but then
with modifications; see below:

Given a Bratteli diagram with vertex-set $V=\left\{ 0\right\} \bigcup_{n=1}^{\infty}V_{n}$,
and vertices $V_{n}$ corresponding to levels $n=1,2,\cdots$ (see
Fig \ref{fig:trp}), we then have the following transition matrices:
\begin{equation}
\begin{split}\begin{cases}
p^{+}\left(n\right){}_{x,y} & x\in V_{n},\: y\in V_{n+1},\;\mbox{and}\\
p^{-}\left(n\right)_{x,z} & x\in V_{n},\: z\in V_{n-1}.
\end{cases}\end{split}
\label{eq:btd1}
\end{equation}
Therefore, in computing transition-probabilities, 
\begin{equation}
\mbox{Prob}\left(x\longrightarrow y\;\mbox{in}\;2m\;\mbox{iterations}\right),\label{eq:btd2}
\end{equation}
we specialize to $x\in V_{n}$, and $y\in V_{n+2k}$. Rather than
the easy formulas $\binom{2m}{m+k}p_{+}^{m+k}p_{-}^{m-k}$ from the
proof in Example \ref{exa:btd}, we now instead get a sum of products
of non-commutative matrices: 
\begin{equation}
P_{w_{1}}P_{w_{2}}\cdots P_{w_{2m}}\label{eq:btd3}
\end{equation}
where $w=\left(w_{1},w_{2},\cdots,w_{2m}\right)$ is a finite word
in the two-letter alphabet $\left\{ \pm\right\} $, i.e., $w_{i}\in\left\{ \pm\right\} $;
but the estimates from before carry over; and we again arrive at an
expression for the Green's function $\left(G_{P}\right)_{x,y}$, $x,y\in V$,
analogous to (\ref{eq:bd13})-(\ref{eq:ta4}).\end{rem}
\begin{example}[The $N$-ary tree]
\label{exa:Ntree} Fix $N>1$. Let $b\in\mathbb{R}_{+}$, $b>1$,
be fixed, and set
\begin{equation}
c\left(n\right):=b^{n},\quad x\in V_{n},\: y\in V_{n+1};\label{eq:btd4}
\end{equation}
then (see \ref{eq:btd1}), we have (see Fig \ref{fig:Ntree}):
\begin{equation}
\begin{split}\begin{cases}
p^{+}\left(n\right)_{xy} & =\dfrac{b}{1+Nb},\vspace{0.5em}\\
p^{-}\left(n\right)_{xz} & =\dfrac{1}{1+Nb},\;\mbox{and}\vspace{0.5em}\\
c\left(n\right)_{x} & =b^{n-1}\left(1+Nb\right)
\end{cases}\end{split}
\label{eq:btd5}
\end{equation}
where $x\in V_{n}$, $y\in V_{n+1}$, $z\in V_{n-1}$.

Generalizing (\ref{eq:ta3}), we get 
\begin{equation}
\left(G_{P}\right)_{x,x'}=\frac{Nb+1}{Nb-1}\label{eq:btd6}
\end{equation}
for all $x,x'\in V_{n}$; and 
\begin{equation}
d_{res}\left(\emptyset,x\right)=\frac{1}{\left(1+Nb\right)b^{n-1}};\label{eq:btd7}
\end{equation}
and $d_{res}\left(x,B\right)<\infty$.

One can show that, if $\#V_{1}<\#V_{2}<\cdots$ (strictly increasing),
then 
\begin{equation}
\dim\left\{ f\::\:\Delta f=0\right\} =\infty.
\end{equation}

\end{example}
\begin{figure}
\includegraphics[width=0.5\textwidth]{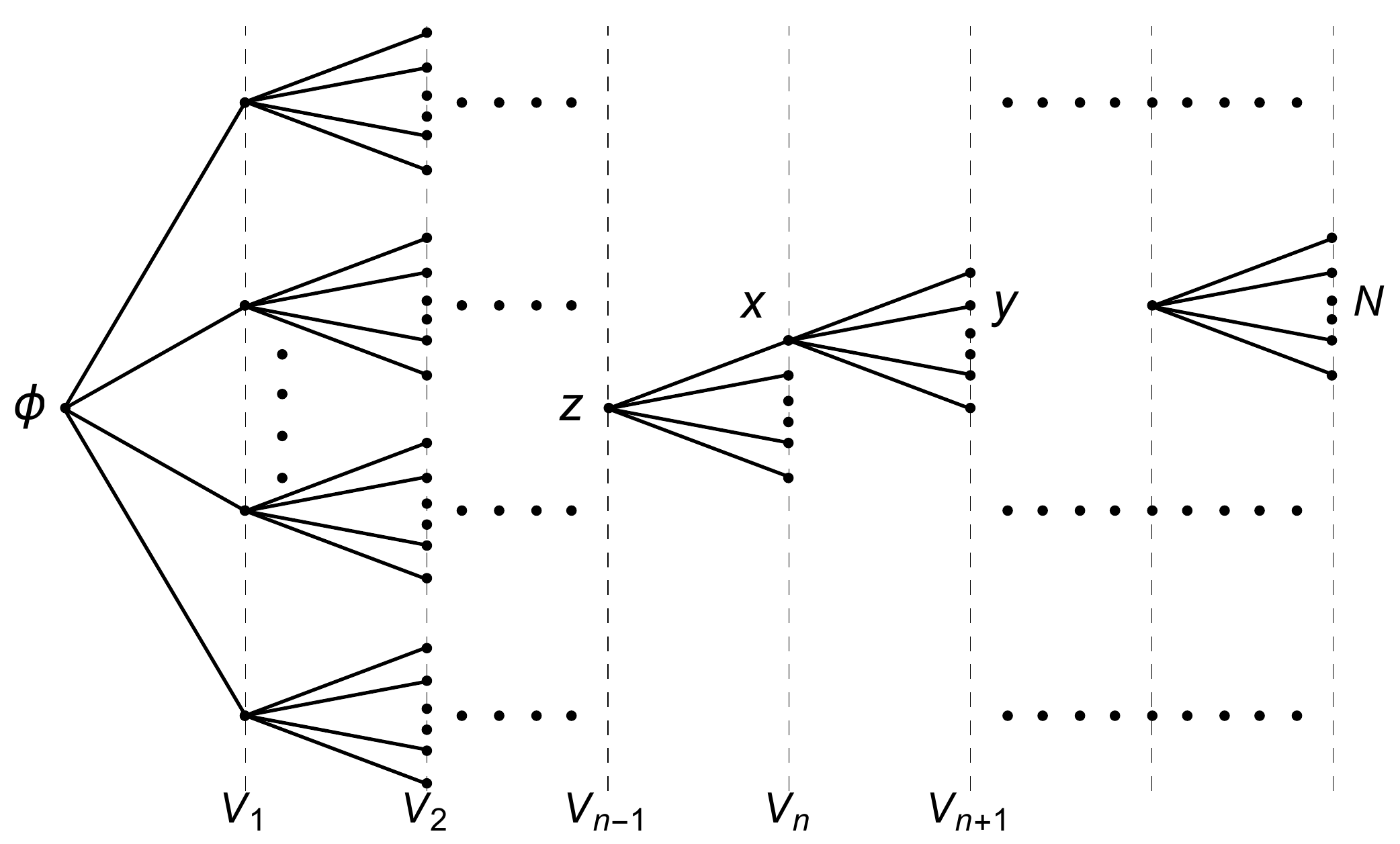}

\protect\caption{\label{fig:Ntree}$N$-ary tree; the vertices at level $n$ are denoted
$V_{n}$, $n=0,1,\cdots$, $V_{0}=\left\{ \emptyset\right\} $, the
empty word. }
\end{figure}

\section{\label{sec:path}The path-space Markov measure vs the Poisson-measure
on $B$}

Here, we consider a class of models $\left(V,E,c\right)$:
\begin{enumerate}[label=(\roman{enumi}),ref=\roman{enumi}]
\item \label{enu:ps1}$V=$ vertices;
\item $E\subset V\times V\backslash\left(\text{diagonal}\right)$ edges; 
\item $c:E\rightarrow\mathbb{R}_{+}$ conduction function, which induces
a resistance metric $d_{res}$; 
\item $\Delta$ graph Laplacian;
\item $\mathscr{H}_{E}$ the energy-Hilbert space; 
\item $B=M\backslash V$, where $M$ is the metric completion; 
\item path space $\Omega=\left\{ \omega=\left(\omega_{i}\right)\:\big|\:\omega_{i}\in V,\:\left(\omega_{i}\omega_{i+1}\right)\in E,\:\forall i\in\mathbb{N}\right\} $;
\item Set $\pi_{i}\left(\omega\right)=\omega_{i}\in V$ (vertex at time
$i$), $i=0,1,2,\cdots$, and 
\[
\Omega_{x}=\left\{ \omega\in\Omega\:\big|\:\pi_{0}\left(\omega\right)=x\right\} ;
\]

\item \label{enu:ps9}Set $p_{xy}=c_{xy}/c\left(x\right)$, $\left(xy\right)\in E$;
\item \label{enu:ps10}$\mu_{x}^{\left(M\right)}$: Markov measure on $\Omega_{x}$,
$x\in V$ with transition
\begin{equation}
\mu_{x}^{\left(M\right)}\left(\text{cylinder}\right)=p_{x\omega_{1}}p_{\omega_{1}\omega_{2}}\cdots\;(\text{see \eqref{enu:ps9}}).\label{eq:cylinder}
\end{equation}

\end{enumerate}
In more detail, a cylinder set $\subset\Omega$ is specified by a
finite word $\left(xx_{1}x_{2}\cdots x_{n}\right)$ of vertices such
that $\left(xx_{1}\right),\left(x_{1}x_{2}\right),\cdots$ are edges
(i.e., in $E$). Then set 
\begin{equation}
C_{xx_{1}\cdots x_{n}}=\left\{ \omega\in\Omega\:\big|\:\pi_{0}\left(\omega\right)=x,\:\pi_{i}\left(\omega\right)=x_{i}\:1\leq i\leq n\right\} .\label{eq:cy1}
\end{equation}
Formula (\ref{eq:cylinder}) then reads as follows:
\begin{equation}
\mu_{x}^{\left(M\right)}\left(C_{xx_{1}x_{2}\cdots x_{n}}\right)=p_{xx_{1}}p_{x_{1}x_{2}}\cdots p_{x_{n-1}}p_{x_{n}}\label{eq:cy2}
\end{equation}
The following is known, see e.g., \cite{MR0405603,MR2358531}:
\begin{lem}
There is a 1-1 correspondence between harmonic functions $h$ on $V$,
on the one hand, and shift-invariant $L^{1}$-functions $F$ on $\Omega$,
on the other. It is given as follows:

Let $\mathbb{E}$ denote the expectation computed w.r.t. the Markov-measure
on $\Omega$, then 
\begin{equation}
h\left(x\right)=\mathbb{E}\left(F\big|\pi_{0}=x\right),\quad x\in V,\label{eq:cy3}
\end{equation}
is harmonic of finite energy iff there is a shift-invariant $L^{1}$-function
$F$ on $\Omega$ such that (\ref{eq:cy3}) holds. (In (\ref{eq:cy3}),
the symbol $\mathbb{E}\left(\cdot\cdot\:|\:\pi_{0}=x\right)$ refers
to conditional expectation.)\end{lem}
\begin{proof}
(see \cite{MR2358531}) use the formula
\begin{equation}
\left(\Delta h\right)\left(x\right)=c\left(x\right)\left(h\left(x\right)-\left(\mathbb{P}h\right)\left(x\right)\right),\quad x\in V,\label{eq:cy4}
\end{equation}
where 
\begin{equation}
\left(\mathbb{P}f\right)\left(x\right)=\sum_{y\sim x}p_{xy}f\left(y\right),\label{eq:cy5}
\end{equation}
and $p_{xy}=c_{xy}/c\left(x\right)$ for $\left(x,y\right)\in E$.\end{proof}
\begin{defn}
\emph{Class $A$} $\left(V,E,c,d_{res}\right)$: 
\begin{equation}
\lim_{k,l\rightarrow\infty}d_{res}\left(\pi_{k}\left(\omega\right),\pi_{l}\left(\omega\right)\right)=0\label{eq:ps6}
\end{equation}
for all $\omega\in\Omega$, or in a ``big'' subset of $\Omega$.\end{defn}
\begin{rem}
A large subset of Bratteli diagram will be of \emph{class A}, i.e.,
that (\ref{eq:ps6}) holds; for example, if 
\begin{equation}
\sum_{n}r\left(n\right)<\infty\label{eq:ps7}
\end{equation}
where $r\left(n\right)$ denotes the resistance $V_{n}\rightarrow V_{n+1}$
between vertices of level $n$ and level $n+1$. So (\ref{eq:ps7})$\Longrightarrow$(\ref{eq:ps6});
but (\ref{eq:ps6}) holds much more generally.\end{rem}
\begin{prop}
\label{prop:ps}Assume (\ref{eq:ps6}). Then there is a well defined
mapping: $\Omega\xrightarrow{\;\Phi\;}B$, given by
\begin{align}
\Omega\longrightarrow\left(\text{Cauchy-sequences}\right)\longrightarrow\left(\text{Cauchy-sequences}\right)/ & \sim\label{eq:ps9}
\end{align}
\begin{align}
\omega\longmapsto\Psi\left(\omega\right) & =\mbox{class}\left(\pi_{0}\left(\omega\right),\pi_{1}\left(\omega\right),\pi_{2}\left(\omega\right),\cdots\right)\label{eq:ps8}
\end{align}
where $\sim$ on Cauchy-sequences 
\begin{equation}
\left(\widetilde{x}\right)\sim\left(\widetilde{y}\right)\underset{\text{Def}}{\Longleftrightarrow}\lim_{i\rightarrow\infty}d_{res}\left(x_{i},y_{i}\right)=0.\label{eq:ps10}
\end{equation}
\end{prop}
\begin{thm}
\label{thm:Mar}Let $V,E,c,p_{xy}=c_{xy}/c\left(x\right)$, $\mu_{x}^{\left(M\right)}$
Markov measure, and let $\Psi:\Omega\rightarrow B$ be the mapping
in (\ref{eq:ps8}) of Proposition \ref{prop:ps}. Then 
\begin{equation}
\left\{ \mu_{x}^{\left(M\right)}\circ\Psi^{-1}\right\} _{x\in V}\label{eq:ps11}
\end{equation}
constitutes the Poisson-measure on $B$ in Theorem \ref{thm:poisson};
i.e., if $S\in\mathscr{B}\left(B\right)$, $S\subset B$ is a given
Borel subset, then the measure in (\ref{eq:ps11}) is $\mu_{x}^{\left(M\right)}\left(\Psi^{-1}\left(S\right)\right)$,
where
\[
\Psi^{-1}\left(S\right)=\left\{ \omega\in\Omega\:\big|\:\Psi\left(\omega\right)\in S\right\} .
\]
\end{thm}
\begin{proof}
(sketch) Set $\mu_{x}:=\mu_{x}^{\left(M\right)}\circ\Psi^{-1}$, we
then need to prove that 
\begin{equation}
h\left(x\right)=\int_{B}\widetilde{h}\: d\mu_{x}\label{eq:ps12}
\end{equation}
holds for all harmonic function $h\in\mathscr{H}_{E}$, i.e., $\left\Vert h\right\Vert _{\mathscr{H}_{E}}<\infty$,
$\Delta h=0$ ($\Longleftrightarrow\mathbb{P}h=h$), and where $\widetilde{h}\in C\left(B\right)$
is the restriction to $B$ of the extension from 
\[
\underset{h}{V}\longrightarrow\underset{\widetilde{h}}{M}\longrightarrow\underset{\widetilde{h}\big|_{B}}{B}
\]
With this, we can check directly that $\mu_{x}$ satisfies (\ref{eq:ps12}),
and so $\mu_{x}$ must be the Poisson-measure by uniqueness.
\end{proof}

\section{\label{sec:bint}Boundary and interpolation}
\begin{thm}
\label{thm:bp}Let $V,E,c,\Delta,d_{res},\mathscr{H}_{E}$, and $B$
be as above (see (\ref{enu:ps1})-(\ref{enu:ps10}) in Section \ref{sec:path}).
We pick a base-point $o\in V$, and dipoles $v_{x}=v_{\left(xo\right)}$
s.t. $v_{x}\left(o\right)=0$, and we set 
\begin{equation}
K\left(x,y\right)=\left\langle v_{x},v_{y}\right\rangle _{\mathscr{H}_{E}}=v_{x}\left(y\right)=v_{y}\left(x\right),\label{eq:bp1}
\end{equation}
the Green's function for $\Delta$. Finally, set $Q:=Q_{Harm}$ denote
the projection of $\mathscr{H}_{E}$ onto the subspace $Harm=\left\{ h\in\mathscr{H}_{E}\:|\:\Delta h=0\right\} $.
For $x\in V$, let $\mu_{x}$ denote the Poisson-measure. 

Then we have the following interpolation/boundary formula:
\begin{equation}
f\left(x\right)=\sum_{y\in V\backslash\left\{ o\right\} }K\left(x,y\right)\left(\Delta f\right)\left(y\right)+\int_{B}\widetilde{\left(Qf\right)}\left(b\right)\: d\mu_{x}\left(b\right)\label{eq:bp2}
\end{equation}
valid for all $f\in\mathscr{H}_{E}$, and all $x\in V$. \end{thm}
\begin{proof}
From \cite{AJLM13,Jor11}, we have that the projection $Q^{\perp}=I_{\mathscr{H}_{E}}-Q$
is given by 
\begin{equation}
\left(Q^{\perp}f\right)=\sum_{y\in V}\left(\Delta f\right)\left(y\right)v_{y}=\sum_{y\in V}\underset{\text{Dirac-notation}}{\underbrace{\left|v_{y}\left\rangle \right\langle \delta_{y}\right|}}\left(f\right);\label{eq:bp3}
\end{equation}
or equivalently,
\begin{equation}
\left(Q^{\perp}f\right)\left(x\right)=\sum_{y\in V\backslash\left\{ o\right\} }K\left(x,y\right)\left(\Delta f\right)\left(y\right),\quad\forall x\in V.\label{eq:bp4}
\end{equation}
Since $f=\left(Q^{\perp}f\right)+\left(Qf\right)$ with $Qf\in Harm\left(\subset\mathscr{H}_{E}\right)$,
the desired formula (\ref{eq:bp2}) follows from the Poisson-representation:
\[
\left(Qf\right)\left(x\right)=\int_{B}\widetilde{\left(Qf\right)}\left(b\right)\: d\mu_{x}\left(b\right).
\]
We have used the following:\end{proof}
\begin{lem}
The operator $A=Q^{\perp}$ in (\ref{eq:bp3}) indeed is a projection
in $\mathscr{H}_{E}$, i.e., $A^{2}=A=A^{*}$ where the adjoint \uline{$*$}
is computed w.r.t. the $\mathscr{H}_{E}$-inner product. \end{lem}
\begin{proof}
We have $A=\sum_{x}\left|v_{x}\left\rangle \right\langle \delta_{x}\right|$,
and so 
\begin{align*}
A^{2} & =\sum_{x}\sum_{y}\left(\left|v_{x}\left\rangle \right\langle \delta_{x}\right|\right)\left(\left|v_{y}\left\rangle \right\langle \delta_{y}\right|\right)\\
 & =\sum_{x}\sum_{y}\left\langle \delta_{x},v_{y}\right\rangle _{\mathscr{H}_{E}}\left|v_{x}\left\rangle \right\langle \delta_{y}\right|\\
 & =\sum_{x}\sum_{y}\delta_{xy}\left|v_{x}\left\rangle \right\langle \delta_{y}\right|=\sum_{x}\left|v_{x}\left\rangle \right\langle \delta_{x}\right|=A.
\end{align*}
But we also have for $f,g\in\mathscr{H}_{E}$, that 
\begin{align*}
\left\langle f,Ag\right\rangle _{\mathscr{H}_{E}} & =\sum_{x}\overline{f\left(x\right)}\left(\Delta g\right)\left(x\right)\\
 & =\sum_{x}\overline{\left(\Delta f\right)\left(x\right)}g\left(x\right)=\left\langle Af,g\right\rangle _{\mathscr{H}_{E}},
\end{align*}
where we use Lemma \ref{lem:Delta} (\ref{enu:D1}), so $A=A^{*}$. 

From this, we get operator-norm $\left\Vert A\right\Vert _{\mathscr{H}_{E}\rightarrow\mathscr{H}_{E}}=1$.
It is immediate from (\ref{eq:bp3}) that $Ah=0$ for all $h\in Harm$,
and further that $A=\text{proj}$ onto $\mathscr{H}_{E}\ominus Harm$.
Recall $\mathscr{H}_{E}\ominus Harm=\mathscr{H}_{E}$-norm closure
of $\left\{ \delta_{x}\:|\: x\in V\right\} $.\end{proof}
\begin{rem}
Note that the function $K\left(\cdot,\cdot\right)$ from (\ref{eq:bp1})-(\ref{eq:bp2})
is a Green's function of the Laplacian $\Delta$. Recall $\Delta$
from Lemma \ref{lem:Delta} has the following $\infty\times\infty$
matrix-representation:
\begin{equation}
\Delta_{xy}=\begin{cases}
c\left(x\right) & \text{if \ensuremath{x=y}}\\
-c_{xy} & \text{if \ensuremath{x\neq y}but \ensuremath{\left(xy\right)\in E}}\\
0 & \text{otherwise};
\end{cases}\label{eq:lg1}
\end{equation}
i.e., as an $\infty\times\infty$ matrix $\Delta$ has the following
banded form
\begin{equation}
\Delta=\begin{pmatrix}\ddots & \ddots\\
\ddots & \ddots & -c_{xz} &  & \substack{\scalebox{2.5}{0}}
\\
 & -c_{xz} & c\left(x\right) & -c_{xy}\\
\substack{\scalebox{2.5}{0}}
 &  & -c_{xy} & \ddots & \ddots\\
 &  &  & \ddots & \ddots
\end{pmatrix}=C-E\label{eq:lg2}
\end{equation}
where $C=\mbox{diag}\left\{ c\left(x\right)\:|\: x\in V\right\} $
\begin{equation}
C=\begin{pmatrix}c\left(x\right) &  & \substack{\scalebox{2.5}{0}}
\\
 & \ddots\\
\substack{\scalebox{2.5}{0}}
 &  & c\left(x'\right)\\
 &  &  & \ddots
\end{pmatrix}\label{eq:lg3}
\end{equation}
and 
\begin{equation}
E=\begin{pmatrix}0 & \ddots\\
\ddots & \ddots & c_{xy}\\
 & c_{xy} & 0 & \ddots\\
 &  & \ddots & \ddots
\end{pmatrix};\label{eq:lg4}
\end{equation}
both $\infty\times\infty$ matrices, $C$ with the sequence $\left(c\left(x\right)\right)_{x\in V}$
as diagonal entries, and $E$ with $0$ in the diagonal but with $\left(c_{xy}\right)$
as entries in the off diagonal entries.

One checks from Lemma \ref{lem:Delta}, that the Green's inversion
then holds:
\begin{equation}
\sum_{z\in V'}\Delta_{xz}K\left(z,y\right)=\delta_{x,y},\quad\forall\left(x,y\right)\in V'\times V'\label{eq:lg5}
\end{equation}
where $K\left(\cdot,\cdot\right)$ in (\ref{eq:lg5}) is the $\infty\times\infty$
matrix introduced in (\ref{eq:bp1}). So infimum about the resistance
metric results from an inversion of the matrix $\left(\Delta_{xy}\right)$
in (\ref{eq:lg2}) above.\end{rem}
\begin{cor}
For every $f\in\mathscr{H}_{E}$ with $f\left(o\right)=0$, we have
the following representation:
\begin{equation}
\left\Vert f\right\Vert _{\mathscr{H}_{E}}^{2}=\left\langle f,\Delta f\right\rangle _{l^{2}}+\int_{B_{Markov}}\big|\widetilde{Qf}\big|^{2}d\mu^{\left(Markov\right)}\label{eq:bp5}
\end{equation}
where 
\begin{equation}
\left\langle f,\Delta f\right\rangle _{l^{2}}=\sum_{x\in V}\overline{f\left(x\right)}\left(\Delta f\right)\left(x\right)\label{eq:bp6}
\end{equation}
and where $\mu^{\left(Markov\right)}$ is the Markov measure from
Theorem \ref{thm:Mar}.\end{cor}
\begin{proof}
First, by Theorem \ref{thm:bp} we have $f=Q^{\perp}f+Qf$ as an orthogonal
splitting, relative to the $\mathscr{H}_{E}$-inner product. Hence
\begin{equation}
\left\Vert f\right\Vert _{\mathscr{H}_{E}}^{2}=\left\Vert Q^{\perp}f\right\Vert _{\mathscr{H}_{E}}^{2}+\left\Vert Qf\right\Vert _{\mathscr{H}_{E}}^{2}.\label{eq:bp7}
\end{equation}
For the first term in (\ref{eq:bp7}), we have 
\begin{eqnarray*}
\left\Vert Q^{\perp}f\right\Vert _{\mathscr{H}_{E}}^{2} & = & \left\langle f,Q^{\perp}f\right\rangle _{\mathscr{H}_{E}}\\
 & \underset{\left(\text{by }\eqref{eq:bp3}\right)}{=} & \sum_{x}\left(\Delta f\right)\left(x\right)\left\langle f,v_{x}\right\rangle _{\mathscr{H}_{E}}\\
 & = & \sum_{x}\overline{f\left(x\right)}\left(\Delta f\right)\left(x\right)=\left\langle f,\Delta f\right\rangle _{l^{2}}.
\end{eqnarray*}
For the second term in (\ref{eq:bp7}), we get, using Proposition
\ref{prop:ps} and Theorem \ref{thm:Mar}, 
\begin{equation}
\left\Vert Qf\right\Vert _{\mathscr{H}_{E}}^{2}=\int_{B_{Markov}}\big|\widetilde{Qf}\big|^{2}d\mu^{\left(Markov\right)};\label{eq:bp8}
\end{equation}
see also \cite{Anc90}. The desired conclusion (\ref{eq:bp5}) now
follows.\end{proof}
\begin{acknowledgement*}
The co-authors thank the following colleagues for helpful and enlightening
discussions: Professors Daniel Alpay, Sergii Bezuglyi, Ilwoo Cho,
Ka Sing Lau, Daniel Lenz, Paul Muhly, Myung-Sin Song, Wayne Polyzou,
Keri Kornelson, and members in the Math Physics seminar at the University
of Iowa.

\bibliographystyle{amsalpha}
\bibliography{ref}

\newcommand{\etalchar}[1]{$^{#1}$}
\providecommand{\bysame}{\leavevmode\hbox to3em{\hrulefill}\thinspace}
\providecommand{\MR}{\relax\ifhmode\unskip\space\fi MR }
\providecommand{\MRhref}[2]{%
  \href{http://www.ams.org/mathscinet-getitem?mr=#1}{#2}
}
\providecommand{\href}[2]{#2}
\begin{thebibliography}{GHK{\etalchar{+}}14}

\bibitem[AJ15]{AJ15}
D.~{Alpay} and P.~{Jorgensen}, \emph{{Reproducing kernel Hilbert spaces
  generated by the binomial coefficients. Accepted. To appear.}}, to appear in
  the Illinois Journal of Mathematics (2015).

\bibitem[AJLM13]{AJLM13}
Daniel Alpay, Palle Jorgensen, Izchak Lewkowicz, and Itzik Marziano,
  \emph{Representation formulas for {H}ardy space functions through the {C}untz
  relations and new interpolation problems}, Multiscale signal analysis and
  modeling, Springer, New York, 2013, pp.~161--182. \MR{3024468}

\bibitem[AJSV13]{AJSV13}
Daniel Alpay, Palle Jorgensen, Ron Seager, and Dan Volok, \emph{On discrete
  analytic functions: products, rational functions and reproducing kernels}, J.
  Appl. Math. Comput. \textbf{41} (2013), no.~1-2, 393--426. \MR{3017129}

\bibitem[AJV14]{AJV14}
Daniel Alpay, Palle Jorgensen, and Dan Volok, \emph{Relative reproducing kernel
  {H}ilbert spaces}, Proc. Amer. Math. Soc. \textbf{142} (2014), no.~11,
  3889--3895. \MR{3251728}

\bibitem[AK12]{MR3025713}
Sergio Albeverio and Seiichiro Kusuoka, \emph{Diffusion processes in thin tubes
  and their limits on graphs}, Ann. Probab. \textbf{40} (2012), no.~5,
  2131--2167. \MR{3025713}

\bibitem[Anc90]{Anc90}
A.~Ancona, \emph{Th\'eorie du potentiel sur les graphes et les vari\'et\'es},
  \'{E}cole d'\'et\'e de {P}robabilit\'es de {S}aint-{F}lour {XVIII}---1988,
  Lecture Notes in Math., vol. 1427, Springer, Berlin, 1990, pp.~1--112.
  \MR{1100282 (92g:31012)}

\bibitem[BJKR00]{MR1804950}
Ola Bratteli, Palle E.~T. J{\o}rgensen, Ki~Hang Kim, and Fred Roush,
  \emph{Non-stationarity of isomorphism between {AF} algebras defined by
  stationary {B}ratteli diagrams}, Ergodic Theory Dynam. Systems \textbf{20}
  (2000), no.~6, 1639--1656. \MR{1804950 (2001k:46104)}

\bibitem[BJO04]{MR2030387}
Ola Bratteli, Palle E.~T. Jorgensen, and Vasyl{\cprime}
  Ostrovs{\cprime}ky{\u\i}, \emph{Representation theory and numerical
  {AF}-invariants. {T}he representations and centralizers of certain states on
  {$\mathscr{O}_d$}}, Mem. Amer. Math. Soc. \textbf{168} (2004), no.~797,
  xviii+178. \MR{2030387 (2005i:46069)}

\bibitem[BKY14]{MR3150704}
S.~Bezuglyi, J.~Kwiatkowski, and R.~Yassawi, \emph{Perfect orderings on finite
  rank {B}ratteli diagrams}, Canad. J. Math. \textbf{66} (2014), no.~1,
  57--101. \MR{3150704}

\bibitem[Bra72]{MR0312282}
Ola Bratteli, \emph{Inductive limits of finite dimensional {$C^{\ast}
  $}-algebras}, Trans. Amer. Math. Soc. \textbf{171} (1972), 195--234.
  \MR{0312282 (47 \#844)}

\bibitem[BV14]{MR3224442}
Christian Bayer and Bezirgen Veliyev, \emph{Utility maximization in a binomial
  model with transaction costs: a duality approach based on the shadow price
  process}, Int. J. Theor. Appl. Finance \textbf{17} (2014), no.~4, 1450022,
  27. \MR{3224442}

\bibitem[CXY15]{MR3290453}
Xiao Chang, Hao Xu, and Shing-Tung Yau, \emph{Spanning trees and random walks
  on weighted graphs}, Pacific J. Math. \textbf{273} (2015), no.~1, 241--255.
  \MR{3290453}

\bibitem[DJ07]{MR2358531}
Dorin~Ervin Dutkay and Palle E.~T. Jorgensen, \emph{Martingales, endomorphisms,
  and covariant systems of operators in {H}ilbert space}, J. Operator Theory
  \textbf{58} (2007), no.~2, 269--310. \MR{2358531 (2009h:47040)}

\bibitem[DJ11]{MR2811284}
\bysame, \emph{Affine fractals as boundaries and their harmonic analysis},
  Proc. Amer. Math. Soc. \textbf{139} (2011), no.~9, 3291--3305. \MR{2811284
  (2012e:28008)}

\bibitem[DJS12]{MR2863855}
Dorin~Ervin Dutkay, Palle E.~T. Jorgensen, and Sergei Silvestrov,
  \emph{Decomposition of wavelet representations and {M}artin boundaries}, J.
  Funct. Anal. \textbf{262} (2012), no.~3, 1043--1061. \MR{2863855}

\bibitem[Doo72]{MR0405603}
J.~L. Doob, \emph{The structure of a {M}arkov chain}, Proceedings of the
  {S}ixth {B}erkeley {S}ymposium on {M}athematical {S}tatistics and
  {P}robability ({U}niv. {C}alifornia, {B}erkeley, {C}alif., 1970/1971), {V}ol.
  {III}: {P}robability theory, Univ. California Press, Berkeley, Calif., 1972,
  pp.~131--141. \MR{0405603 (53 \#9395)}

\bibitem[DS88]{DS88b}
Nelson Dunford and Jacob~T. Schwartz, \emph{Linear operators. {P}art {II}},
  Wiley Classics Library, John Wiley \& Sons Inc., New York, 1988, Spectral
  theory. Selfadjoint operators in Hilbert space, With the assistance of
  William G. Bade and Robert G. Bartle, Reprint of the 1963 original, A
  Wiley-Interscience Publication. \MR{1009163 (90g:47001b)}

\bibitem[GHK{\etalchar{+}}14]{Georgakopoulos2014}
Agelos Georgakopoulos, Sebastian Haeseler, Matthias Keller, Daniel Lenz, and
  Rados{\l}aw~K. Wojciechowski, \emph{Graphs of finite measure}, Journal de
  Math{\'e}matiques Pures et Appliqu{\'e}es (2014), no.~0, --.

\bibitem[GP14]{MR3158629}
Igor Gorodezky and Igor Pak, \emph{Generalized loop-erased random walks and
  approximate reachability}, Random Structures Algorithms \textbf{44} (2014),
  no.~2, 201--223. \MR{3158629}

\bibitem[GPS99]{MR1710743}
Thierry Giordano, Ian~F. Putnam, and Christian~F. Skau, \emph{Full groups of
  {C}antor minimal systems}, Israel J. Math. \textbf{111} (1999), 285--320.
  \MR{1710743 (2000g:46096)}

\bibitem[Her12]{MR2982692}
Sa'ar Hersonsky, \emph{Boundary value problems on planar graphs and flat
  surfaces with integer cone singularities, {I}: {T}he {D}irichlet problem}, J.
  Reine Angew. Math. \textbf{670} (2012), 65--92. \MR{2982692}

\bibitem[HPS92]{MR1194074}
Richard~H. Herman, Ian~F. Putnam, and Christian~F. Skau, \emph{Ordered
  {B}ratteli diagrams, dimension groups and topological dynamics}, Internat. J.
  Math. \textbf{3} (1992), no.~6, 827--864. \MR{1194074 (94f:46096)}

\bibitem[J{\o}r81]{Jor81}
Palle E.~T. J{\o}rgensen, \emph{A uniqueness theorem for the
  {H}eisenberg-{W}eyl commutation relations with nonselfadjoint position
  operator}, Amer. J. Math. \textbf{103} (1981), no.~2, 273--287. \MR{610477
  (82g:81033)}

\bibitem[Jor08]{MR2432048}
Palle E.~T. Jorgensen, \emph{Essential self-adjointness of the
  graph-{L}aplacian}, J. Math. Phys. \textbf{49} (2008), no.~7, 073510, 33.
  \MR{2432048 (2009k:47099)}

\bibitem[Jor11]{Jor11}
\bysame, \emph{A sampling theory for infinite weighted graphs}, Opuscula Math.
  \textbf{31} (2011), no.~2, 209--236. \MR{2747308 (2012d:05271)}

\bibitem[JP10]{JP10}
Palle E.~T. Jorgensen and Erin Peter~James Pearse, \emph{A {H}ilbert space
  approach to effective resistance metric}, Complex Anal. Oper. Theory
  \textbf{4} (2010), no.~4, 975--1013. \MR{2735315 (2011j:05338)}

\bibitem[JP11]{JP11}
Palle E.~T. Jorgensen and Erin P.~J. Pearse, \emph{Resistance boundaries of
  infinite networks}, Random walks, boundaries and spectra, Progr. Probab.,
  vol.~64, Birkh\"auser/Springer Basel AG, Basel, 2011, pp.~111--142.
  \MR{3051696}

\bibitem[KL12]{Keller_2012}
Matthias Keller and Daniel Lenz, \emph{Dirichlet forms and stochastic
  completeness of graphs and subgraphs}, Journal f{\"u}r die reine und
  angewandte Mathematik (Crelles Journal) \textbf{2012} (2012), no.~666.

\bibitem[KM67]{MR0225999}
Yukio Kusunoki and Shin'ichi Mori, \emph{Some remarks on boundary values of
  harmonic functions with finite {D}irichlet integrals}, J. Math. Kyoto Univ.
  \textbf{7} (1967), 315--324. \MR{0225999 (37 \#1589)}

\bibitem[KPS12]{MR2905788}
Vadim Kostrykin, J{\"u}rgen Potthoff, and Robert Schrader, \emph{Brownian
  motions on metric graphs}, J. Math. Phys. \textbf{53} (2012), no.~9, 095206,
  36. \MR{2905788}

\bibitem[Rob11]{MR2815030}
Thomas Roblin, \emph{Comportement harmonique des densit\'es conformes et
  fronti\`ere de {M}artin}, Bull. Soc. Math. France \textbf{139} (2011), no.~1,
  97--128. \MR{2815030 (2012f:31011)}

\bibitem[Rud91]{MR1157815}
Walter Rudin, \emph{Functional analysis}, second ed., International Series in
  Pure and Applied Mathematics, McGraw-Hill Inc., New York, 1991. \MR{MR1157815
  (92k:46001)}

\bibitem[Saw97]{MR1463727}
Stanley~A. Sawyer, \emph{Martin boundaries and random walks}, Harmonic
  functions on trees and buildings ({N}ew {Y}ork, 1995), Contemp. Math., vol.
  206, Amer. Math. Soc., Providence, RI, 1997, pp.~17--44. \MR{1463727
  (98k:60127)}

\bibitem[Shi83]{MR674604}
Hiroshige Shiga, \emph{On the quasiconformal deformation of open {R}iemann
  surfaces and variations of some conformal invariants}, J. Math. Kyoto Univ.
  \textbf{22} (1982/83), no.~3, 463--480. \MR{674604 (84j:30082)}

\bibitem[Sko13]{MR3046303}
M.~Skopenkov, \emph{The boundary value problem for discrete analytic
  functions}, Adv. Math. \textbf{240} (2013), 61--87. \MR{3046303}

\bibitem[TB13]{MR3039515}
J.~Tosiek and P.~Brzykcy, \emph{States in the {H}ilbert space formulation and
  in the phase space formulation of quantum mechanics}, Ann. Physics
  \textbf{332} (2013), 1--15. \MR{3039515}

\bibitem[{Woj}07]{2007PhDT.......216W}
R.~K. {Wojciechowski}, \emph{{Stochastic Completeness of Graphs}}, Ph.D.
  thesis, PhD Thesis, 2007, 2007.

\end{thebibliography}
\end{acknowledgement*}

\end{document}